\documentclass[letterpaper,10pt]{IEEEtran}
\usepackage{url}

\usepackage{amsmath,mathrsfs, textcomp, mathtools,amsthm}
\usepackage{color}
\usepackage{amssymb}
\usepackage[all]{xy}
\usepackage[caption = false]{subfig}
\usepackage{graphicx,xspace,bm}
\usepackage{pdfsync}
\usepackage[lined,linesnumbered,algoruled,noend]{algorithm2e}

\newtheorem{theorem}{Theorem}[section]

\newtheorem{assumption}[theorem]{Assumption}
\newtheorem{lemma}[theorem]{Lemma}

\newtheorem{remark}[theorem]{Remark}

\newtheorem{corollary}[theorem]{Corollary}
\newtheorem{proposition}[theorem]{Proposition}  
\newcommand{\Deg}{\mathsf{D}}

\newcommand{\Adj}{\mathsf{A}}
\newcommand{\Lap}{\mathsf{L}}
\newcommand{\vertices}{\mathcal{V}}
\newcommand{\edges}{\mathcal{E}}

\newcommand{\Bgraph}{\mathcal{G}}

\newcommand{\norm}[1]{\ensuremath{\| #1 \|}}

\newcommand{\real}{{\mathbb{R}}}
\newcommand{\realex}{\overline{\mathbb{R}}}

\newcommand{\realnonnegative}{{\mathbb{R}}_{\ge 0}}

\newcommand{\integerspositive}{\mathbb{Z}_{\geq 1}}

\newcommand{\identity}[1]{\mathsf{I}_{#1}}

\newcommand{\Lie}{\mathcal{L}}

\newcommand{\gradient}{\nabla}
\newcommand{\setdef}[2]{\{#1 \; | \; #2\}}
\newcommand{\setdefbig}[2]{\big\{#1 \; | \; #2\big\}}
\newcommand{\eps}{\epsilon}
\newcommand{\argmin}{\operatorname{argmin}}
\newcommand{\until}[1]{\{1,\dots,#1\}}
\newcommand{\map}[3]{#1:#2 \rightarrow #3}

\newcommand{\epi}{\mathrm{epi}}

\newcommand{\BB}{\mathcal{B}}
\newcommand{\CC}{\mathcal{C}}
\newcommand{\II}{\mathcal{I}}

\renewcommand{\SS}{\mathcal{S}}
\newcommand{\KK}{\mathcal{K}}

\newcommand{\YY}{\mathcal{Y}}
\newcommand{\FF}{\mathcal{F}}
\newcommand{\GG}{\mathcal{G}}

\newcommand{\HH}{\mathcal{H}}
\newcommand{\NN}{\mathcal{N}}

\newcommand{\EE}{\mathcal{E}}
\newcommand{\MM}{\mathcal{M}}
\newcommand{\PP}{\mathcal{P}}
\newcommand{\VV}{\mathcal{V}}
\newcommand{\XX}{\mathcal{X}}

\newcommand{\intr}{\mathrm{int}}
\newcommand{\ri}{\mathrm{ri}}
\newcommand{\aff}{\mathrm{aff}}
\newcommand{\cl}{\mathrm{cl}}

\newcommand{\ones}{\mathbf{1}}
\newcommand{\zeros}{\mathbf{0}}

\newcommand{\Eb}{\mathbb{E}}
\newcommand{\Pb}{\mathbb{P}}
\newcommand{\Qb}{\mathbb{Q}}

\newcommand{\lm}{\lambda}

\newcommand{\Xihat}{\widehat{\Xi}}
\newcommand{\xihat}{\widehat{\xi}}
\newcommand{\data}[1]{\xihat^{\,#1}}
\newcommand{\xival}[1]{\xi^{#1}}
\newcommand{\xibval}[1]{\overline{\xi}^{\, #1}}
\newcommand{\xval}[1]{x^{#1}}
\newcommand{\lmval}[1]{\lambda^{#1}}
\newcommand{\nuval}[1]{\nu^{#1}}
\newcommand{\etaval}[1]{\eta^{#1}}
\newcommand{\what}{\widehat{w}}
\newcommand{\yhat}{\widehat{y}}

\newcommand{\xhat}{\widehat{x}}
\newcommand{\Jhat}{\widehat{J}}
\newcommand{\PPhat}{\widehat{\PP}}
\newcommand{\Pbhat}{\widehat{\Pb}}

\newcommand{\st}{\operatorname{subject \text{$\, \,$} to}}

\newcommand{\lmvec}{\lambda_{\operatorname{v}}}
\newcommand{\xvec}{x_{\operatorname{v}}}
\newcommand{\yvec}{y_{\operatorname{v}}}

\newcommand{\xo}{x_{*}}
\newcommand{\yo}{y_{*}}
\newcommand{\zo}{z_{*}}
\newcommand{\xt}{\tilde{x}}
\newcommand{\yt}{\tilde{y}}

\newcommand{\xit}{\tilde{\xi}}

\newcommand{\xvb}{\overline{x}_{\operatorname{v}}}
\newcommand{\lmvb}{\overline{\lm}_{\operatorname{v}}}
\newcommand{\xvecb}{\overline{x}_{\operatorname{v}}} 
\newcommand{\lmvecb}{\overline{\lm}_{\operatorname{v}}} 

\newcommand{\Laug}{L_{\operatorname{aug}}}
\newcommand{\Laugt}{\tilde{L}_{\operatorname{aug}}}
\newcommand{\Laugr}{\overline{L}_{\operatorname{aug}}}

\newcommand{\spLaugt}{X_{\text{sp}}}

\renewcommand{\until}[1]{[#1]}
\newcommand{\proj}{\mathrm{proj}}

\newcommand{\cvar}{\mathrm{CVaR}}
\newcommand{\val}{\mathrm{val}}
\newcommand{\rad}{\eps}
\newcommand{\rads}{\eps^2}

\renewcommand{\theenumi}{(\roman{enumi})}

\renewcommand{\baselinestretch}{1}

\newcommand{\oprocendsymbol}{\hbox{$\bullet$}}
\newcommand{\oprocend}{\relax\ifmmode\else\unskip\hfill\fi\oprocendsymbol}

\newcommand{\longthmtitle}[1]{\mbox{}\textup{\textsl{(#1):}}}

\allowdisplaybreaks

\begin{document}

\title{Cooperative data-driven distributionally robust
  optimization\thanks{A preliminary version of this work appeared at
    the 2017 Allerton Conference on Communication, Control, and
    Computing, Monticello, Illinois as~\cite{AC-JC:17-allerton}.}}

\author{Ashish Cherukuri \qquad Jorge Cort\'{e}s\thanks{A. Cherukuri,
    J.~Cort\'{e}s are with the Dept. of~Mechanical and Aerospace
    Engineering, UC San Diego,
    \texttt{\{acheruku,cortes\}@ucsd.edu}.}}

\maketitle 

\begin{abstract}
  This paper studies a class of multiagent stochastic optimization
  problems where the objective is to minimize the expected value of a
  function which depends on a random variable.  
  The probability
  distribution of the random variable is unknown to the agents, so
  each one gathers samples of it. 
  The agents then aim to cooperatively
  find, using their data, a solution to the optimization problem with
  guaranteed out-of-sample performance. The approach is to formulate a
  data-driven distributionally robust optimization problem using
  Wasserstein ambiguity sets, which turns out to be equivalent to a
  convex program. We reformulate the latter as a distributed
  optimization problem and identify a convex-concave augmented
  Lagrangian function whose saddle points are in correspondence with
  the optimizers provided a min-max interchangeability criteria is
  met. Our distributed algorithm design then consists of the
  saddle-point dynamics associated to the augmented Lagrangian. We
  formally establish that the trajectories of the dynamics converge
  asymptotically to a saddle point and hence an optimizer of the
  problem.  Finally, we provide a class of functions that meet the
  min-max interchangeability criteria. Simulations illustrate our
  results.
\end{abstract}

\section{Introduction}\label{sec:intro}

Stochastic optimization in the context of multiagent systems has
numerous applications, such as target tracking, distributed
estimation, and cooperative planning and learning.  Solving
  stochastic optimization problems, in an exact sense,
  requires the knowledge of the probability distribution of the random
  variables. Even then, computing this optimizer is computationally
  burdensome because of the expectation operator. To mitigate this
  problem, researchers have studied numerous sample-based methods that
  provide tractable ways of approximating the optimizer. One of the
  concerns of such methods is obtaining out-of-sample performance,
  avoiding overfitting. The concern is more pressing when only a few
  samples are available, typically in applications where acquiring
  samples is expensive due to the size and complexity of the system or
  when decisions must be taken in real time, leaving less room for
  gathering many samples.  Distributionally robust optimization (DRO)
  provides a regularization framework that guarantees good
  out-of-sample performance even when the data is disturbed and not
  sampled from the true distribution. Motivated by this, we consider
  here the task for a group of agents to collaboratively find a
  data-driven solution for a stochastic optimization problem using the
  tools provided by the DRO framework.

\emph{Literature review:} Stochastic optimization is a classical
topic~\cite{AS-DD-AR:14}. To the large set of methods available to
solve this type of problems, a recent addition is data-driven
distributionally robust optimization, see
e.g.,~\cite{PME-DK:18, DB-VG-NK:18, RG-AJK:16-arXiv,FL-SM:17-arXiv, CZ:14-thesis} and
references therein. In this setup, the distribution of the random
variable is unknown and so, a worst-case optimization is carried over
a set of distributions, termed ambiguity set.
This worst-case optimization provides probabilistic performance bounds
for the original stochastic optimization~\cite{PME-DK:18,
  JB-YK-KM:17-arXiv} and overcomes the problem of overfitting.  One
way of designing the ambiguity sets is to consider the set of
distributions that are close (in some distance metric over the space
of distributions) to some reference distribution constructed from the
available data. Depending on the metric, one gets different ambiguity
sets with different performance bounds. Some popular metrics are
$\phi$-divergence~\cite{RJ-YG:16}, Prohorov metric~\cite{EE-GI:06},
and Wasserstein distance~\cite{PME-DK:18}. Here, we consider ambiguity
sets defined using the Wasserstein metric.
In~\cite{DB-VG-NK:18}, the ambiguity set is constructed with
  distributions that pass a goodness-of-fit test. In addition to
  data-driven methods, other works on distributionally robust
  optimization consider ambiguity sets defined using moment
  constraints~\cite{ED-YY:10, WW-DK-MS:14} and the KL-divergence
  distance~\cite{ZH-LJH:13}.  Tractable reformulations for the
data-driven DRO methods have been well
studied~\cite{PME-DK:18,RG-AJK:16-arXiv,JB-KM:17-arXiv}.  However,
designing coordination algorithms to find a data-driven solution when
the data is gathered in a distributed way by a network of agents has
not been investigated. This is the focus of this paper.
Our work has connections with the growing body of literature on
distribution optimization problems~\cite{DPB-JNT:97,MGR-RDN:05,PW-MDL:09} 
and agreement-based algorithms to solve them, see
e.g.,~\cite{AN-AO:09,BJ-MR-MJ:09,MZ-SM:12,JW-NE:11,AN:14-sv,BG-JC:14-tac}
and references therein. 

Besides data-driven DRO, one can solve the stochastic
  optimization problem considered here via other sampling-based
  methods, see~\cite{THM-GB:14}. Among these, sample average
  approximation (SAA) and stochastic approximation (SA) have received
  much attention because of their simple implementation and
  finite-sample guarantees independent of the dimension of the
  uncertainty, see e.g.~\cite[Chapter 5]{AS-DD-AR:14}
  and~\cite{AN-AJ-GL-AS:09}. However, such guarantees need not hold
  when the samples are corrupted and may require stricter assumptions
  on the cost function and the feasibility set. In contrast, the
  sample guarantees of the data-driven DRO method hold for more
  general settings, see e.g.,~\cite{PME-DK:18,JB-YK-KM:17-arXiv}, but
  are (potentially) more conservative and do not scale well with the
  size of the uncertainty parameter. Additionally, the complexity of
  solving a data-driven DRO is often worse than that of the SAA and SA
  methods.

\emph{Statement of contributions:} Our starting point is a multiagent
stochastic optimization problem involving the minimization of the
expected value of an objective function with a decision variable and a
random variable as arguments. The probability distribution of the
random variable is unknown and instead, agents collect a finite set of
samples of it. Given this data, each agent can individually find a
data-driven solution of the stochastic optimization.  However, agents
wish to cooperate to leverage on the data collected by everyone in the
group.  Our approach consists of formulating a distributionally robust
optimization problem over ambiguity sets defined as neighborhoods of
the empirical distribution under the Wasserstein metric. The solution
of this problem has guaranteed out-of-sample performance for the
stochastic optimization.  Our first contribution is the reformulation
of the DRO problem to display a structure amenable to distributed
algorithm design. We achieve this by augmenting the decision variables
to yield a convex optimization whose objective function is the
aggregate of individual objectives and whose constraints involve
consensus among neighboring agents.  Building on an augmented version
of the associated Lagrangian function, we identify a convex-concave
function which under a min-max interchangeability condition has the
property that its saddle-points are in one-to-one correspondence with
the optimizers of the reformulated problem.  Our second contribution
is the design of the saddle-point dynamics for the identified
convex-concave Lagrangian function. We show that the proposed dynamics
is distributed and provably correct, in the sense that its
trajectories asymptotically converge to a solution of the original
stochastic optimization problem.  Our third contribution is the
identification of two broad class of objective functions for which the
min-max interchangeability holds.  The first class is the set of
functions that are convex-concave in the decision and the random
variable, respectively. The second class is where functions are
convex-convex and have some additional structure: they are either
quadratic in the random variable or they correspond to the loss
function of the least-squares problem.  Finally, we illustrate our
results in simulation.

\section{Preliminaries}\label{sec:prelims}

This section introduces notation and basic notions on graph theory,
convex analysis, and stability of discontinuous dynamical systems. A reader already familiar with these concepts can safely skip it.

\subsubsection{Notation}\label{subsec:notation}
Let $\real$, $\realnonnegative$, and $\integerspositive$ denote the set
of real, nonnegative real, and positive integer numbers. The extended
reals are denoted as $\realex = \real \cup \{-\infty,\infty\}$. 
For a positive integer $n \in \integerspositive$, the set $\until{n} := \{1, \dots , n\}$.
We let
$\norm{\cdot}$ denote the $2$-norm on $\real^n$. We use the notation
$B_\delta(x):=\setdef{y \in \real^n}{\norm{x-y} < \delta}$.
Given $x,y\in \real^n$, $x_i$ denotes the $i$-th component of
$x$, and $x \le y$ denotes $x_i \le y_i$ for $i \in \until{n}$. For
vectors $u \in \real^n$ and $w \in \real^m$, the vector $(u;w) \in
\real^{n+m}$ denotes their concatenation.  We use the shorthand notation
$\zeros_n = (0,\ldots,0) \in \real^n$, $\ones_n=(1,\ldots,1) \in
\real^n$, and $\identity{n} \in \real^{n \times n}$ for the identity
matrix. For $A \in \real^{n_1 \times n_2}$ and $B \in \real^{m_1 \times
m_2}$, $A \otimes B \in \real^{n_1 m_1 \times n_2 m_2}$ is the Kronecker
product. The Cartesian product of $\{\SS_i\}_{i=1}^n$ is denoted by
$\prod_{i=1}^n \SS_i := \SS_1 \times \dots \times \SS_n$. The interior
of a set $\SS \subset \real^n$ is denoted by $\intr(\SS)$.  For a
function $f:\real^n \times \real^m \to \real$, $(x,\xi) \mapsto
f(x,\xi)$, we denote the partial derivative of $f$ with respect to the
first argument by $\gradient_x f$ and with respect to the second
argument by $\gradient_\xi f$.  The higher-order derivatives follow the
convention $\gradient_{x \xi} f = \frac{\partial^2 f}{\partial x
\partial \xi}$, $\gradient_{xx} f = \frac{\partial^{2} f}{\partial
x^2}$, and so on. Given $\map{V}{\XX}{\realnonnegative}$, we denote the
$\delta$-sublevel set as $V^{-1}(\le \delta):=\setdef{x \in \XX}{V(x)
\le \delta}$.

\subsubsection{Graph theory}\label{subsec:graph}
Following~\cite{FB-JC-SM:08cor}, an \emph{undirected graph}, or simply
a \emph{graph}, is a pair $\GG=(\VV, \EE)$, where $\VV = \until{n}$ is
the vertex set and $\EE \subseteq \VV \times \VV$ is the edge set with
the property that $(i,j) \in \EE$ if and only if $(j,i) \in \EE$. A
path is an ordered sequence such that any ordered pair of vertices
appearing consecutively is an edge. A graph is \emph{connected} if
there is a path between any pair of distinct vertices.  Let $\NN_i
\subseteq \VV$ denote the set of neighbors of vertex $i \in \VV$,
i.e., $\NN_i = \setdef{j \in \VV}{(i,j) \in \edges}$.  A
\emph{weighted graph} is a triplet $\Bgraph=(\VV,\EE,\Adj) $, where
$(\VV,\EE)$ is a digraph and $\Adj \in \mathbb{R}^{n\times n}_{\geq0}$
is the (symmetric) \emph{adjacency matrix} of $\Bgraph$, with the
property that $a_{ij}>0 $ if $ (i,j)\in \edges $ and $a_{ij}=0$,
otherwise. 
The \emph{weighted degree} of $i \in \until{n}$ is $w_i = \sum_{j=1}^n
a_{ij}$.  The \emph{weighted degree} matrix $\Deg$ is the diagonal
matrix defined by $(\Deg)_{ii} = w_i$, for all $i \in \until{n}$.  The
\emph{Laplacian} matrix is $\Lap = \Deg - \Adj$. Note that $\Lap =
\Lap^\top$ and $\Lap \ones_n = 0$. If $\GG$ is connected, then zero is
a simple eigenvalue of $\Lap$.

\subsubsection{Convex analysis}\label{subsec:convex-a}

Here we introduce elements from convex analysis
following~\cite{RTR:97}.  A set $C \subset \real^n$ is \emph{convex}
if $(1-\lm)x+\lm y \in C$ whenever $x \in C$, $y \in C$, and $\lm \in
(0,1)$. A vector $\varphi \in \real^n$ is \emph{normal} to a convex
set $C$ at a point $x \in C$ if $(y - x)^\top \varphi \le 0$ for all
$y \in C$. The set of all vectors normal to $C$ at $x$, denoted
$N_C(x)$, is the \emph{normal cone} to $C$ at $x$. The \emph{affine
  hull} of $S \subset \real^n$ is the smallest affine space
containing~$S$,
\begin{align*}
  \aff(S):=\setdefbig{\sum_{i=1}^k \lm_i x_i}{k \in \integerspositive,
    x_i \in S, \lm_i \in \real, \sum_{i=1}^k \lm_i = 1}.
\end{align*}
The \emph{relative interior} of a convex set $C$ is the interior of
$C$ relative to the affine hull of $C$. Formally,
\begin{align*}
  \ri(C) \! := \! \setdef{x \! \in \! \aff(C)}{\exists \eps > 0, (x+\eps
  B_1(0)) \cap (\aff(C) ) \! \subset C}.
\end{align*}
Given a convex set $C$, a vector $d$ is a \emph{direction of
  recession} of $C$ if $x+\alpha d \in C$ for all $x \in C$ and
$\alpha \ge 0$.

A convex function $\map{f}{\real^n}{\realex}$ is \emph{proper} if
there exists $x \in \real^n$ such that $f(x) < +\infty$ and $f$ does
not take the value $-\infty$ anywhere in $\real^n$. The
\emph{epigraph} of $f$ is the set
\begin{align*}
  \epi f := \setdef{(x,\lm) \in (\real^n \times \realex)}{\lm \ge
    f(x)}.
\end{align*}
A function $f$ is closed if $\epi f$ is a closed set.  The function
$f$ is convex if and only if $\epi f$ is convex. For a closed proper
convex function $f$, a vector $d$ is a \emph{direction of recession}
of $f$ if $(d,0)$ is a direction of recession of the set $\epi
f$. Intuitively, it is the direction along which $f$ is monotonically
non-increasing. If $f(x) \to +\infty$ whenever $\norm{x} \to +\infty$,
then $f$ does not have a direction of recession.

A function $\map{F}{\XX \times \YY}{\realex}$ is \emph{convex-concave}
(on $\XX \times \YY$) if, given any point $(\xt,\yt) \in \XX \times
\YY$, $x \mapsto F(x,\yt)$ is convex and $y \mapsto F(\xt,y)$ is
concave. When the space $\XX \times \YY$ is clear from the context, we
refer to this property as $F$ being convex-concave in $(x,y)$. A point
$(\xo,\yo) \in \XX \times \YY$ is a \emph{saddle point} of $F$ over the
set $\XX \times \YY$ if $F(\xo,y) \le F(\xo,\yo) \le F(x,\yo)$, for all
$x \in \XX$ and $y \in \YY$. The set of saddle points of a
convex-concave function~$F$ is convex.  Each saddle point is a critical
point of $F$, i.e., if $F$ is differentiable, then $\gradient_x
F(\xo,\zo) =0$ and $\gradient_z F(\xo,\zo) =0$.  Additionally, if $F$ is
twice differentiable, then $\gradient_{xx} F(\xo,\zo) \succeq 0$ and
$\gradient_{zz} F(\xo,\zo) \preceq 0$.  Given a convex-concave function
$\map{F}{\real^n \times \real^m}{\realex}$, define 
\begin{align*} 
  \XX & := \setdef{x \in \real^n}{F(x,y) < +\infty \text{ for all } y
\in \real^m},
  \\
  \YY & := \setdef{y \in \real^m}{F(x,y) > - \infty \text{ for all } x
\in \real^n}.
\end{align*}
The product set $\XX \times \YY$ is called the \emph{effective domain}
of $F$. The sets $\XX$, $\YY$ and so, $\XX \times \YY$ are convex. Note
that $F$ is finite on $\XX \times \YY$. If $\XX \times \YY$ is nonempty,
then $F$ is called \emph{proper}. If the
following equality holds
\begin{align*}
  \sup_{y \in \YY} \inf_{x \in \XX} F(x,y) = \inf_{x \in \XX} \sup_{y
  \in \YY} F(x,y),
\end{align*}
then this common value is called the \emph{saddle value} of $F$. The
function $F$ is
\emph{closed} if for any $(x,y) \in \XX \times \YY$, the functions
$x \mapsto F(x,y)$ and $y \mapsto -F(x,y)$ are closed. 

\begin{theorem}\longthmtitle{Existence of finite saddle value and saddle
    point~\cite[Theorem 37.3 \& 37.6]{RTR:97}}\label{th:saddle-value}
  Let $\map{F}{\real^n \times \real^m}{\realex}$ be a closed proper
  convex-concave function with effective domain $\XX \times \YY \subset
  \real^n \times \real^m$. 
  If the following conditions hold,
  \begin{enumerate}
    \item The convex functions $x \mapsto F(x,y)$ for $y \in \ri(\YY)$ have no
      common direction of recession;
    \item The convex functions $y \mapsto -F(x,y)$ for $x \in \ri(\XX)$ have
      no common direction of recession;
  \end{enumerate}
  then the saddle value must be finite, there exists a saddle point of
  $F$ in the effective domain $\XX \times \YY$ and the saddle value is
  attained at the saddle point.
\end{theorem}

\subsubsection{Discontinuous dynamical systems}\label{subsec:disc}

Here we present notions of discontinuous and projected dynamical systems
from~\cite{AB-FC:06,JC:08-csm-yo,AN-DZ:96}. Let $\map{f}{\real^n}{\real^n}$ be
a Lebesgue measurable and locally bounded function, and consider
\begin{equation}\label{eq:dis-dyn}
  \dot x = f(x) .
\end{equation}
A map $\map{\gamma}{[0,T)}{\real^n}$ is a \emph{(Caratheodory)
  solution} of~\eqref{eq:dis-dyn} on the interval $[0,T)$ if it is
absolutely continuous on $[0,T)$ and satisfies $\dot \gamma(t) =
f(\gamma(t))$ almost everywhere in $[0,T)$.  We use the terms solution
and trajectory interchangeably.  A set $\SS \subset \real^n$ is
\emph{invariant} under~\eqref{eq:dis-dyn} if every solution starting
in $\SS$ remains in $\SS$.  For a solution $\gamma$
of~\eqref{eq:dis-dyn} defined on the time interval $[0,\infty)$, the
\emph{omega-limit} set $\Omega(\gamma)$ is defined by
\begin{multline*}
  \Omega(\gamma) = \setdef{y \in \real^n}{\text{there exists } 
    \{t_k\}_{k=1}^{\infty} \subset [0,\infty) \text{ with } 
      \\
      \lim_{k \to \infty} t_k = \infty \text{ and } \lim_{k \to \infty}
    \gamma(t_k) = y} .
\end{multline*}
If the solution $\gamma$ is bounded, then $\Omega (\gamma) \neq
\emptyset$ by the Bolzano-Weierstrass theorem~\cite[p. 33]{SL:93}.
Given a continuously differentiable function
$\map{V}{\real^n}{\real}$, the \emph{Lie derivative of $V$
  along~\eqref{eq:dis-dyn}} at $x \in \real^n$ is $\Lie_f V(x) =
\gradient V(x)^\top f(x)$.  The next result is a simplified version
of~\cite[Proposition 3]{AB-FC:06}.

\begin{proposition}\longthmtitle{Invariance principle for
    discontinuous Caratheodory systems}\label{pr:invariance-cara}
  Let $\SS \subset \real^n$ be compact and invariant. Assume that, for
  each point $x_0 \in \SS$, there exists a unique solution
  of~\eqref{eq:dis-dyn} starting at $x_0$ and that its omega-limit set
  is invariant too. Let $\map{V}{\real^n}{\real}$ be a continuously
  differentiable map such that $\Lie_f V(x) \le 0$ for all $x \in
  \SS$. Then, any solution of~\eqref{eq:dis-dyn} starting at $\SS$
  converges to the largest invariant set in $\cl(\setdef{x \in
    \SS}{\Lie_f V(x) = 0})$.
\end{proposition}

Projected dynamical systems are a particular class of discontinuous
dynamical systems. Let $\KK \subset \real^n$ be a closed convex set.  Given a
point $y \in \real^n$, the (point) projection of $y$ onto $\KK$ is
$\proj_{\KK}(y) = \argmin_{z \in \KK} \norm{z - y}$. Note that
$\proj_{\KK}(y)$ is a singleton and the map $\proj_{\KK}$ is Lipschitz
on $\real^n$ with constant $L = 1$~\cite[Proposition 2.4.1]{FHC:83}.
Given $x \in \KK$ and $v \in \real^n$, the (vector) projection of $v$
at $x$ with respect to~$\KK$~is
\begin{equation*}
  \Pi_{\KK}(x,v) = \lim_{\delta \to 0^+} \frac{\proj_{\KK}(x+\delta v) -
    x}{\delta} .
\end{equation*}
Given a vector field $\map{f}{\real^n}{\real^n}$ and a closed convex
polyhedron $\KK \subset \real^n$, the associated projected dynamical
system is
\begin{equation}\label{eq:pds}
  \dot x = \Pi_{\KK}(x,f(x)), \quad x(0) \in \KK,
\end{equation}
One can verify easily that for any $x \in \KK$, there exists an element
$\varphi_x$ belonging to the normal cone $N_{\KK}(x)$ such that
$\Pi_{\KK}(x,f(x)) = f(x) - \varphi_x$.  In particular, if $x$ is in the
interior of $\KK$, then this element is the zero vector and we have
$\Pi_{\KK}(x,f(x)) = f(x)$.  At any boundary point of $\KK$, the
projection operator restricts the flow of the vector field $f$ such that
the solutions of~\eqref{eq:pds} remain in $\KK$. Due to the projection,
the dynamics~\eqref{eq:pds} is in general discontinuous.   

\section{Data-driven stochastic optimization}\label{sec:data-driven}
This section sets the stage for the formulation of our approach to
deal with data-driven optimization in a distributed manner.  The
following material on data-driven stochastic optimization is taken
from~\cite{PME-DK:18} and included here to provide a
self-contained exposition. The reader familiar with these notions and
tools can safely skip this section.

Let $(\Omega,\FF,P)$ be a probability space and $\xi$ be a random
variable mapping this space to $(\real^m,B_\sigma(\real^m))$, where
$B_\sigma(\real^m)$ is the Borel $\sigma$-algebra on $\real^m$. Let
$\Pb$ and $\Xi \subseteq \real^m$ be the distribution and the support
of the random variable $\xi$. Assume that $\Xi$ is closed and convex.
Consider the stochastic optimization problem
\begin{equation}\label{eq:s-opt}
  \inf_{x \in \XX} \Eb_{\Pb} [f(x,\xi)],
\end{equation}
where $\XX \subseteq \real^n$ is a closed convex set, $\map{f}{\real^n
  \times \real^m}{\real}$ is a continuous function,
and $\Eb_{\Pb}[\, \cdot \,]$ is the expectation under the distribution
$\Pb$. Assume that $\Pb$ is unknown and so, solving~\eqref{eq:s-opt}
is not possible.  However, we are given $N$ independently drawn
samples 
$\Xihat := \{\data{k}\}_{k=1}^N \subset \Xi$ 
of the random
variable~$\xi$. Note that, until it is revealed, $\Xihat$ is a random
object with probability distribution $\Pb^N:=\prod_{i=1}^N \Pb$
supported on $\Xi^N:=\prod_{i=1}^N \Xi$.  The objective is to find a
\emph{data-driven} solution of~\eqref{eq:s-opt}, denoted $\xhat_N \in
\XX$, constructed using the dataset $\Xihat$, that has desirable
properties for the expected cost $\Eb_{\Pb} [f(\xhat_N,\xi)]$ under a
new sample. The property we are looking for is the \emph{finite-sample
  guarantee} given by 
\begin{equation}\label{eq:fs-guarantee}
  \Pb^N \Bigl(\Eb_{\Pb} [f(\xhat_N,\xi)] \le \Jhat_N \Bigr) \ge 1-\beta,
\end{equation}
where $\Jhat_N$ might also depend on the training dataset and $\beta
\in (0,1)$ is the parameter which governs $\xhat_N$ and $\Jhat_N$.
The quantities $\Jhat_N$ and $1-\beta$ are referred to as the
\emph{certificate} and the \emph{reliability} of the performance of
$\xhat_N$. The goal is to find a data-driven solution with a low
certificate and a high reliability. To do so, we use the available
information~$\Xihat$. The strategy is to determine a set $\PPhat_N$
of probability distributions supported on $\Xi$ so that minimization of the worst-case cost over $\PPhat_N$ results into a finite-sample guarantee.
The set $\PPhat_N$ is
referred to as the \emph{ambiguity} set.  Once such a set 
is designed, the certificate $\Jhat_N$ is defined as the optimal value
of the following \emph{distributionally robust optimization} problem
\begin{equation}\label{eq:dro}
  \Jhat_N := \inf_{x \in \XX} \sup_{\Qb \in \PPhat_N} \Eb_{\Qb}
  [f(x,\xi)].
\end{equation}
This is the worst-case optimal value considering all distributions in
$\PPhat_N$.  A good candidate for $\PPhat_N$ is the set of
distributions that are close (under a certain metric) to the uniform
distribution on $\Xihat$, termed the \emph{empirical distribution}.
Formally, the empirical distribution is
\begin{equation}
  \Pbhat_N := \frac{1}{N} \sum_{k=1}^N \delta_{\data{k}},
\end{equation}
where $\delta_{\data{k}}$ is the unit point mass at
$\data{k}$.  Let $\MM(\Xi)$ be the space of probability
distributions $\Qb$ supported on $\Xi$ with finite second moment, i.e.,
$\Eb_{\Qb} [\norm{\xi}^2] = \int_{\Xi} \norm{\xi}^2 \Qb(d \xi) < +\infty$.
The \emph{2-Wasserstein metric}
\footnote{We note that~\cite{PME-DK:18} employs the 1-Wasserstein metric
instead of the 2-Wasserstein metric considered here.
} $\map{d_{W_2}}{\MM(\Xi) \times \MM(\Xi)}{\realnonnegative}$ is
\begin{align}
  d_{W_2}(\Qb_1, \Qb_2)  =  \Bigl( \inf \Bigl\{\int_{\Xi^2}  \norm{\xi_1
  - & \xi_2}^2  \Pi(d \xi_1, d \xi_2) \Bigl| \notag
  \\
  &  \Pi \in \HH(\Qb_1,\Qb_2) \Bigr\} \Bigr)^{\frac{1}{2}},
  \label{eq:wasserstein}
\end{align}
where $\HH(\Qb_1,\Qb_2)$ is the set of all distributions on $\Xi \times
\Xi$ with marginals $\Qb_1$ and $\Qb_2$.  Given $\eps \ge 0$, we use the
notation 
\begin{align}\label{eq:ball-dist}
  \BB_{\eps}(\Pbhat_N) := \setdef{\Qb \in \MM(\Xi)}{d_{W_2}(\Pbhat_N,\Qb)
  \le \eps}
\end{align}
to define the set of distributions that are $\eps$-close to $\Pbhat_N$
under the defined metric. For an appropriately chosen radius
  $\eps$, the ambiguity set $\PPhat_N = \BB_{\rad}(\Pbhat_N)$,
  plugged in the distributionally robust optimization~\eqref{eq:dro},
  results into a finite-sample
  guarantee~\eqref{eq:fs-guarantee}. There might be different ways of
  establishing this fact. For example, in~\cite{PME-DK:18}, a bound
  for $\eps$ is provided under the assumption that $\Pb$ is
  light-tailed satisfying an exponential decay condition. The
  work~\cite{JB-YK-KM:17-arXiv}, on the other hand, considers more
  general distributions and gives a different, potentially tighter,
  finite-sample guarantee. However, in~\cite{JB-YK-KM:17-arXiv}, $f$
  is assumed to be either quadratic or log-exponential loss function.
 The focus of this work is on the design of distributed
  algorithms to solve~\eqref{eq:dro} with $\BB_{\rad}(\Pbhat_N)$ as
  the ambiguity set. To this end, the following tractable
  reformulation is key.
\begin{theorem}\longthmtitle{Tractable reformulation
    of~\eqref{eq:dro}}\label{th:tractable-dro}
  Assume that for all $\xit \in \Xi$, $x \mapsto f(x,\xit)$ is convex.
  Then, for 
  $N \in \integerspositive$, the optimal value of~\eqref{eq:dro} with
  the choice $\PPhat_N = \BB_{\rad}(\Pbhat_N)$ is equal to
  the optimum of the following convex optimization problem
  \begin{equation*}
     \inf_{\lm \ge 0, x \in \XX} \Bigl\{ \lm \rads +
       \frac{1}{N} \sum_{k=1}^N \max_{\xi \in \Xi} \Bigl( f(x,\xi) - \lm
       \norm{\xi - \data{k}}^2 \Bigr) \Bigr\}.
  \end{equation*}
\end{theorem}

This result and its proof are similar to~\cite[Theorem
  4.2]{PME-DK:18} and its corresponding proof, respectively. While our
  metric is $2$-Wasserstein, the referred result's is
  $1$-Wasserstein. Theorem~\ref{th:tractable-dro} shows that under
mild conditions on the objective function, one can reformulate the
distributionally robust optimization problem as a convex optimization
problem. This result plays a key role in our forthcoming
discussion. We note that the reformulation given in
  Theorem~\ref{th:tractable-dro} is valid under weaker set of
  conditions on $f$, as reported in~\cite{RG-AJK:16-arXiv}
  and~\cite{JB-KM:17-arXiv}. We however avoid this generality as it
  complicates the design and analysis of the distributed algorithm.

\section{Problem statement}\label{sec:problem}

Consider $n \in \integerspositive$ agents communicating over an
undirected weighted graph $\GG = (\vertices,\edges,\Adj)$. The set of
vertices are enumerated as $\VV:=\until{n}$.  Each agent $i \in
\until{n}$ can send and receive information from its neighbors $\NN_i$
in $\GG$.  Let $\map{f}{\real^d \times \real^m}{\real}$, $(x,\xi)
\mapsto f(x,\xi)$, be a continuously differentiable objective
function.  Assume that for any $\xi \in \real^m$, the map $x \mapsto
f(x,\xi)$ is convex and that for any $x \in \real^d$, the map $\xi
\mapsto f(x,\xi)$ is either convex or concave. Suppose that the set of
$\xi \in \real^m$ for which $\ones_n$ and $-\ones_n$ are not a
direction of recession for the convex function $x \mapsto f(x,\xi)$ is
dense in $\real^m$.
As we progress, we stipulate additional conditions on $f$ as
necessary. Assume that all agents know the objective
function~$f$. Given a random variable $\xi \in \real^m$ with support
$\real^m$ and distribution $\Pb$, the original objective for the
agents is to solve the following stochastic optimization problem
\begin{equation}\label{eq:stochastic-net-opt-0}
  \inf_{x \in \real^d} \Eb_{\Pb} \bigl[ f(x,\xi) \bigr].
\end{equation}
For simplicity, we optimize over $\real^d$ instead of some closed convex set $\XX$. However, our proposed method can handle such generalization by assuming that each agent knows a subset of $\real^d$ such that the intersection of them all is $\XX$.
We assume that $\Pb$ is unknown to agents and 
instead, each agent has a certain number (at least one) of independent
and identically distributed realizations of the random variable $\xi$.
We denote the data available to agent $i$ by $\Xihat_i$ that is assumed
to be nonempty. Assume that
$\Xihat_i \cap \Xihat_j = \emptyset$ for all $i,j \in \until{n}$ and let
$\Xihat= \cup_{i=1} \Xihat_i$ containing $N$ samples be the available
data set.  

The goal for the agents is then to collectively find, in a distributed
manner, a data-driven solution $\xhat_N \in \real^d$ to approximate the
optimizer of~\eqref{eq:stochastic-net-opt-0} with guaranteed performance
bounds.  To achieve this, we rely on the framework of distributionally
robust optimization, cf.  Section~\ref{sec:data-driven}.  From
Theorem~\ref{th:tractable-dro}, a data-driven solution
for~\eqref{eq:stochastic-net-opt-0} can be obtained by solving the
following convex optimization problem
\begin{align}\label{eq:stochastic-net-c-opt}
  \inf_{\lm \ge 0, x} \Bigl\{ \! \lm \rads \!+\! \frac{1}{N}
  \!  \sum_{k=1}^N \max_{\xi \in \real^m} \Bigl( f(x,\xi) \!-\! \lm
  \norm{\xi - \data{k}}^2 \Bigr) \! \Bigr\}.
\end{align}
The following is assumed to hold throughout the paper. 
\begin{assumption}\longthmtitle{Nontrivial feasibility and existence of finite optimizers of~\eqref{eq:stochastic-net-c-opt}}\label{as:standing}
	\rm{ 
	We assume that the set belonging to
	$\realnonnegative \times \real^d$ where the objective function
	in~\eqref{eq:stochastic-net-c-opt} takes finite values has a
	nonempty interior. Further, we assume that there exists a finite optimizer $(x^*,\lm^*)$
	of~\eqref{eq:stochastic-net-c-opt}.
	\oprocend
}
\end{assumption}
The existence of finite optimizers is ensured if one of the set of conditions for such existence given in~\cite{AEO-PT:06} are met.
Note that each agent can individually find a data-driven solution
to~\eqref{eq:stochastic-net-opt-0} by using only the data available to
it in the convex formulation~\eqref{eq:stochastic-net-c-opt}. However,
such a solution in general will have an inferior out-of-sample
guarantee as compared to the one obtained collectively. In the
cooperative setting, agents aim to
solve~\eqref{eq:stochastic-net-c-opt} in a distributed manner, that is
\begin{enumerate}
\item each agent $i$ has the information
  \begin{equation}\label{eq:info}
	\II_i := \{\Xihat_i, f, \eps, n, N\}, 
  \end{equation}
  where $\eps$ is the radius of the ambiguity set that agents agree upon beforehand,
\item each agent $i$ can only communicate with its neighbors $\NN_i$
  in the graph $\GG$, 
\item each agent $i$ does not share with its neighbors any element of
  the dataset $\Xihat_i$ available to it, and 
\item there is no central coordinator or leader that can communicate
  with all agents. 
\end{enumerate}
The challenge in solving~\eqref{eq:stochastic-net-c-opt} in a
distributed manner lies in the fact that the data is distributed over
the network and the optimizer~$x^*$ depends on it all. Moreover, the
inner maximization can be a nonconvex problem, in general.  One way of
solving~\eqref{eq:stochastic-net-c-opt} in a cooperative fashion is to
let agents share their data with everyone in the network via some sort
of flooding mechanism. This violates item (iii) of our definition of
distributed algorithm given above. We specifically keep such methods
out of scope due to two reasons.  First, the data would not be private
anymore, creating a possibility of adversarial action. Second, the
communication burden of such a strategy is higher than our proposed
distributed strategy when the size of the network and the dataset
grows along the execution of the algorithm.

Our strategy to tackle the problem is organized as follows: in
Section~\ref{sec:reformulation} we reformulate the
problem~\eqref{eq:stochastic-net-c-opt} to obtain a structure which
allows us in Section~\ref{sec:dist} to propose our distributed
algorithm.  Section~\ref{sec:dist1} discusses a class of objective
functions~$f$ for which the distributed algorithm provably converges.

\section{Distributed problem formulation and saddle
  points}\label{sec:reformulation}

This section studies the structure of the optimization problem
presented in Section~\ref{sec:problem} with the ulterior goal of
facilitating the design of a distributed algorithmic solution. Our
first step is a reformulation of~\eqref{eq:stochastic-net-c-opt} that,
by augmenting the decision variables of the agents, yields an
optimization where the objective function is the aggregate of
individual functions (that can be independently evaluated by the
agents) and constraints which display a distributed structure.  Our
second step is the identification of a convex-concave function whose
saddle points are the primal-dual optimizers of the reformulated
problem under suitable conditions on the objective function~$f$.  This
opens the way to consider the associated saddle-point dynamics as our
candidate distributed algorithm. The structure of the original
optimization problem makes this step particularly nontrivial.

  \begin{remark}\longthmtitle{Alternative distributed algorithmic solutions}\label{re:alternate}
    \rm{The optimization problem~\eqref{eq:stochastic-net-c-opt} can
      possibly be solved using other distributed methods. This might
      entail making use of alternative reformulations
      of~\eqref{eq:stochastic-net-c-opt}. For instance,
      problem~\eqref{eq:stochastic-net-c-opt} can be written as a
      semi-infinite program, cf.~\cite{FL-SM:17-arXiv}, and then a
      distributed cutting-surface method can be designed following the
      centralized algorithm given in~\cite{FL-SM:17-arXiv}. When $f$
      is piecewise affine in $\xi$,~\eqref{eq:stochastic-net-c-opt}
      takes the form of a conic program (without the $\max$ operator
      in the objective), which can potentially be solved via
      primal-dual distributed solvers. Finally,
      following~\cite{JB-YK-KM:17-arXiv,RG-XC-AJK:17-arXiv}, for
      certain $f$ (linear form or objective of LASSO or logistic
      regression), the problem~\eqref{eq:stochastic-net-c-opt} is
      equivalent to minimizing the empirical cost (expectation of cost
      function under empirical distribution) plus a regularizer
      term. For such cases, primal-dual distributed solvers may be a
      valid solution strategy. The advantage of the methodology
      proposed here is its generality, which does not require to write
      different algorithms for different cases depending on the form
      of~$f$.  } \oprocend
  \end{remark}

\subsection{Reformulation as distributed optimization problem}

We have each agent $i \in \until{n}$ maintain a copy of $\lm$ and $x$,
denoted by $\lmval{i} \in \real$ and $\xval{i} \in \real^d$,
respectively.  Thus, the decision variables for $i$ are
$(\xval{i},\lmval{i})$. For notational ease, let the concatenated
vectors be $\lmvec := (\lmval{1}; \dots; \lmval{n})$, and $\xvec :=
(\xval{1}; \dots; \xval{n})$.  Let $v_k \in \until{n}$ be the agent
that holds the $k$-th sample $\data{k}$ of the dataset.  Consider the
following convex optimization problem
\begin{subequations}\label{eq:stochastic-net-dist-opt}
  \begin{align}
    \underset{\xvec,\lmvec \ge \zeros_n}{\min} & \quad h(\lmvec)  + \frac{1}{N}
    \sum_{k=1}^N \max_{\xi \in \real^m} g_k(\xval{v_k}, \lmval{v_k},\xi)
    \label{eq:s-n-d-o-1}
    \\
    \st 
    & \quad \Lap \lmvec = \zeros_n, \label{eq:s-n-d-o-2}
    \\
    & \quad (\Lap \otimes \identity{d}) \xvec = \zeros_{nd}, 
    \label{eq:s-n-d-o-3}
  \end{align}
\end{subequations}
where $\Lap \in \real^{n \times n}$ is the Laplacian of the graph $\GG$
and we have used the shorthand notation $\map{h}{\real^n}{\real}$ for
\begin{align*}
  h(\lmvec):= \frac{\rads (\ones_n^\top \lmvec)}{n}
\end{align*}
and, for each $k \in \until{N}$, 
$\map{g_k}{\real^d \times \real \times \real^m}{\real}$ for
\begin{align*}
  g_k(x,\lm,\xi) := f(x,\xi) - \lm \norm{\xi - \data{k}}^2.
\end{align*}

The following result establishes the correspondence between the
optimizers of~\eqref{eq:stochastic-net-c-opt}
and~\eqref{eq:stochastic-net-dist-opt}, respectively.

\begin{lemma}\longthmtitle{One-to-one correspondence between optimizers
    of~\eqref{eq:stochastic-net-c-opt}
    and~\eqref{eq:stochastic-net-dist-opt}}\label{le:corr}
  The following holds:
  \begin{enumerate}
    \item If $(x^*,\lm^*)$ is an optimizer
      of~\eqref{eq:stochastic-net-c-opt}, then $(\ones_n \otimes x^*,
      \lm^* \ones_n)$ is an optimizer
      of~\eqref{eq:stochastic-net-dist-opt}.
    \item If $(\xvec^*,\lmvec^*)$ is an optimizer
      of~\eqref{eq:stochastic-net-dist-opt}, then there exists an
      optimizer $(x^*,\lm^*)$ of~\eqref{eq:stochastic-net-c-opt} such
      that $\xvec^* = \ones_n \otimes x^*$ and $\lmvec^* = \lm^*
      \ones_n$.
  \end{enumerate}
\end{lemma}
\begin{proof}
  The proof follows by noting that $\GG$ is connected and hence, (i)
  $\Lap \lmvec = \zeros_n$ if and only if $\lmvec = \alpha \ones_n$,
  $\alpha \in \real$; and (ii) $(\Lap \otimes \identity{d}) \xvec =
  \zeros_{nd}$ if and only if $\xvec = \ones_n \otimes x$, $x \in
  \real^d$. 
\end{proof}

Note that constraints~\eqref{eq:s-n-d-o-2} and~\eqref{eq:s-n-d-o-3}
force agreement and that each of their components is computable by an
agent of the network using only local information. Moreover, the
objective function~\eqref{eq:s-n-d-o-1} can be written as
$\sum_{i=1}^n J_i (\xval{i},\lmval{i},\Xihat_i)$, where
\begin{align*}
  J_i(\xval{i},\lmval{i},\Xihat_i)
   := \frac{\rads \lmval{i}}{n}
    + \frac{1}{N} \sum_{k: \data{k} \in \Xihat_i} \max_{\xi \in \real^m}
    g_k(\xval{i},\lmval{i},\xi) ,
\end{align*}
for all $i \in \until{n}$.  Therefore, the
problem~\eqref{eq:stochastic-net-dist-opt} has the adequate structure
from a distributed optimization viewpoint: an aggregate objective
function and locally computable constraints.

\subsection{Augmented Lagrangian and saddle points}

Our next step is to identify an appropriate variant of the Lagrangian
function of~\eqref{eq:stochastic-net-dist-opt} with the following two
properties: (i) it does not consist of an inner maximization, unlike
the objective in~\eqref{eq:s-n-d-o-1}, and (ii) the primal-dual
optimizers of~\eqref{eq:stochastic-net-dist-opt} are saddle points of
the newly introduced function. The availability of these two facts
sets the stage for our ensuing algorithm design.

To proceed further, we first denote for convenience the objective
function~\eqref{eq:s-n-d-o-1} with $\map{F}{\real^{nd} \times
  \realnonnegative^n}{\real}$,
\begin{align}
  F(\xvec,\lmvec):= h(\lmvec) + \frac{1}{N} \sum_{k=1}^N \max_{\xi \in
  \real^m} g_k(\xval{v_k},\lmval{v_k},\xi). \label{eq:F}
\end{align}
Note that the Lagrangian of~\eqref{eq:stochastic-net-dist-opt} is
$\map{L}{\real^{nd} \times \realnonnegative^n \times \real^n \times
  \real^{nd}}{\realex}$,
\begin{align}
  L(\xvec,\lmvec,\nu,\eta) \! :=  \! F(\xvec,\lmvec)
  + \nu^\top \Lap \lmvec + \eta^\top (\Lap \otimes
    \identity{d}) \xvec,
  \label{eq:lagrangian}
\end{align}
where $\nu \in \real^n$ and $\eta \in \real^{nd}$ are dual variables
corresponding to the equality constraints~\eqref{eq:s-n-d-o-2}
and~\eqref{eq:s-n-d-o-3}, respectively. $L$ is convex-concave in
$((\xvec,\lmvec),(\nu,\eta))$ on the domain $\lmvec \ge \zeros_n$.
The next result states that the duality gap
for~\eqref{eq:stochastic-net-dist-opt} is zero. The result is a consequence of~\cite[Corollary 28.22]{RTR:97} and~\cite[Theorem 28.3]{RTR:97} using the hypotheses of Assumption~\ref{as:standing}. 

\begin{lemma}\longthmtitle{Min-max equality for $L$}\label{le:zero-duality}
  The set of saddle points of $L$ over the domain $(\real^{nd} \times
  \realnonnegative^n) \times (\real^n \times \real^{nd})$ is nonempty
  and 
  \begin{align}
    \inf_{\xvec,\lmvec \ge \zeros_n} \sup_{\nu,\eta} L
    (\xvec,\lmvec,\nu,\eta)= \sup_{\nu, \eta} \inf_{\xvec,\lmvec \ge
    \zeros_n} L(\xvec,\lmvec,\nu,\eta). \label{eq:min-max-lag}
  \end{align} 
  Furthermore, the following holds:
  \begin{enumerate}
    \item If $(\xvb, \lmvb, \bar \nu , \bar \eta)$ is a saddle point of
    $L$ over $(\real^{nd} \times \realnonnegative^n) \times (\real^n
    \times \real^{nd})$, then $(\xvb, \lmvb)$ is an optimizer
    of~\eqref{eq:stochastic-net-dist-opt}.
    \item If $(\xvb,\lmvb)$ is an optimizer
    of~\eqref{eq:stochastic-net-dist-opt}, then there exists $(\bar \nu,
    \bar \eta)$ such that $(\xvb, \lmvb, \bar \nu , \bar \eta)$ is a
    saddle point of $L$ over $(\real^{nd} \times \realnonnegative^n)
    \times (\real^n \times \real^{nd})$. 
  \end{enumerate}
\end{lemma}

Owing to the above result, one could potentially write a saddle-point
dynamics for the Lagrangian~$L$ as a distributed algorithm to find the
optimizers. However, without strict or strong convexity assumptions on
the objective function, the resulting dynamics is in general not
guaranteed to converge, see e.g.,~\cite{AC-BG-JC:17-sicon}.  To
overcome this hurdle, we augment the Lagrangian with quadratic terms
in the primal variables. Let the augmented Lagrangian
$\map{\Laug}{\real^{nd} \times \realnonnegative^n \times \real^n
  \times \real^{nd}}{\realex}$~be
\begin{align*}
  \Laug(\xvec, \lmvec,\nu,\eta) &:= L(\xvec,\lmvec,\nu,\eta) 
  \\
  & \quad + \frac{1}{2}
  \xvec^\top (\Lap \otimes \identity{d}) \xvec + \frac{1}{2} \lmvec^\top
  \Lap \lmvec.
\end{align*}
Note that $\Laug$ is also convex-concave in
$((\xvec,\lmvec),(\nu,\eta))$ on the domain $\lmvec \ge \zeros_n$.
The next result guarantees that this augmentation step does not change
the saddle points.

\begin{lemma}\longthmtitle{Saddle points of $L$ and $\Laug$ are
    the same}\label{le:equivalence-aug}
  A point $(\xvec^*,\lmvec^*,\nu^*,\eta^*)$ is a saddle point of $L$
  over
  $(\real^{nd} \times \realnonnegative^n) \times (\real^{n} \times
  \real^{nd})$
  if and only if it is a saddle point of $\Laug$ over the same domain.
\end{lemma}
The proof follows by using the convexity property of the objective function 
in~\cite[Theorem 1.1]{XLS-DL-KIMM:05}.
The above result implies that finding the saddle points of~$\Laug$
would take us to the primal-dual optimizers
of~\eqref{eq:stochastic-net-dist-opt}.  However, a final roadblock
remaining is writing a gradient-based dynamics for~$\Laug$, given that
this function involves a set of maximizations in its definition and so
the gradient of $\Laug$ with respect to $\xvec$ is undefined for
$\lmvec = 0$. Thus, our next task is to get rid of these internal
optimization routines and identify a function for which the
saddle-point dynamics is well defined over the feasible domain.  Note
that
\begin{align}
  \Laug(\xvec,\lmvec,\nu,\eta) = \max_{\{\xival{k}\}}
  \Laugt(\xvec,\lmvec,\nu,\eta,\{\xival{k}\}), \label{eq:max-def-laug}
\end{align}
where
\begin{align}
  & \Laugt(\xvec,\lmvec,\nu,\eta,\{\xival{k}\}) := h(\lmvec) +
  \frac{1}{N} \sum_{k=1}^N g_k(\xval{v_k},\lmval{v_k}, \xival{k})
  \notag
  \\
  &  + \nu^\top \Lap \lmvec \! + \! \eta^\top (\Lap \otimes \identity{d})
  \xvec  \! + \!  \frac{1}{2} \xvec^\top (\Lap \otimes \identity{d}) \xvec \! + \! 
  \frac{1}{2} \lmvec^\top \Lap \lmvec.  \label{eq:lagrangian-2}
\end{align}
The following result shows that, under appropriate conditions,
$\Laugt$ is the function we are looking for.

\begin{proposition}\longthmtitle{Saddle points of $\Laugt$ and
    correspondence with optimizers
    of~\eqref{eq:stochastic-net-dist-opt}}
  \label{pr:saddle-primal-dual-equiv}
  Let $\CC \subset \real^{nd} \times \realnonnegative^n$ with
  $\intr(\CC) \not = \emptyset$ be a closed, convex set such that
  \begin{enumerate}
  \item the saddle points of $\Laug$ over the domain $(\real^{nd} \times
  \realnonnegative^n) \times (\real^n \times \real^{nd})$ are contained
  in the set $\CC \times (\real^n
    \times \real^{nd})$;
  \item $\Laugt$ is convex-concave on $\CC \times
    (\real^n \times \real^{nd} \times \real^{mN})$;
  \item for any $(\nu,\eta)$, 
    \begin{align}
      & \min_{ (\xvec,\lmvec) \in \CC}  \max_{ \{\xival{k}\}} 
      \Laugt(\xvec,\lmvec,\nu,\eta,\{\xival{k}\}) \notag
      \\
      & \quad = \max_{\{\xival{k}\}}
      \min_{ (\xvec,\lmvec) \in \CC}
      \Laugt(\xvec,\lmvec,\nu,\eta,\{\xival{k}\}).
      \label{eq:min-max-exchange}
    \end{align}
  \end{enumerate}
  Then, the following holds
  \begin{enumerate}
  \item The set of saddle points of $\Laugt$ over the domain $\CC
    \times (\real^n \times \real^{nd} \times \real^{mN})$ is
    nonempty, convex, and closed.
  \item If $(\xvb,\lmvb,\bar \nu, \bar \eta,\{( \bar \xi)^{k}\})$ is
    a saddle point of $\Laugt$ over $\CC \times (\real^n \times
    \real^{nd} \times \real^{mN})$, then $(\xvb, \lmvb)$ is an
    optimizer of~\eqref{eq:stochastic-net-dist-opt}.
  \item If $(\xvb,\lmvb) \in \CC$ is an optimizer
    of~\eqref{eq:stochastic-net-dist-opt}, then there exists $(\bar
    \nu, \bar \eta, \{(\bar \xi)^{k}\})$ such that $(\xvb,\lmvb,\bar
    \nu, \bar \eta,\{( \bar \xi)^{k}\})$ is a saddle point of $\Laugt$
    over $\CC \times (\real^n \times \real^{nd} \times \real^{mN})$.
  \end{enumerate}
\end{proposition}
\begin{proof}
  If saddle points of $\Laug$ belong to $\CC \times (\real^n \times
  \real^{nd})$, then according to~\cite[Lemma 36.2]{RTR:97}, we have
  \begin{align*}
    \min_{(\xvec,\lmvec) \in \CC} & \max_{\nu,\eta}
    \Laug(\xvec,\lmvec,\nu,\eta) 
    \\
    & = \max_{\nu, \eta} \min_{(\xvec,\lmvec)
      \in \CC} \Laug(\xvec,\lmvec,\nu,\eta).
  \end{align*}
  Using the definition~\eqref{eq:max-def-laug} of $\Laug$ in the above
  equality, we get
  \begin{align}
    & \min_{(\xvec,\lmvec) \in \CC} \max_{\nu,\eta} 
    \max_{\{\xival{k}\}} \Laugt (\xvec,\lmvec,\nu,\eta,\{\xival{k}\})
    \notag
    \\
    & \quad = \max_{\nu, \eta} \min_{(\xvec,\lmvec) \in \CC} 
    \max_{\{\xival{k}\}} \Laugt(\xvec,\lmvec,\nu,\eta,\{\xival{k}\}).
    \label{eq:min-max-max2}
  \end{align}
  Using~\eqref{eq:min-max-exchange} on the right-hand side of the 
  above expression gives 
  \begin{align*}
    \min_{(\xvec,\lmvec) \in \CC} & \max_{\nu,\eta, \{\xival{k}\}}
    \Laugt(\xvec,\lmvec,\nu,\eta,\{\xival{k}\})  
    \\
    & = \max_{\nu,\eta, \{\xival{k}\}} \min_{(\xvec,\lmvec) \in \CC}
    \Laugt(\xvec,\lmvec,\nu,\eta,\{\xival{k}\}).
  \end{align*}
  From the above equality and the fact that $\Laugt$ is convex-concave
  and finite-valued, we conclude from~\cite[Lemma 36.2]{RTR:97} that
  the set of saddle points of $\Laugt$ over the domain $\CC \times
  (\real^n \times \real^{nd} \times \real^{mN})$ is nonempty. Further,
  this set is closed and convex again due to convexity-concavity of
  $\Laugt$.  Finally, parts (ii) and (iii) follow from combining
  Lemmas~\ref{le:zero-duality} and~\ref{le:equivalence-aug} with the
  following two facts.  First, from~\eqref{eq:min-max-max2}, if
  $(\xvb,\lmvb,\bar \nu, \bar \eta, \{(\bar \xi)^k\})$ is a saddle
  point of $\Laugt$, then $(\xvb,\lmvb,\bar \nu, \bar \eta)$ is a
  saddle point of $\Laug$.  Second, if $(\xvb,\lmvb,\bar \nu, \bar
  \eta)$ is a saddle point of $\Laug$, then there exists $\{ (\bar
  \xi)^k\}$, which is the maximizer of $\{ \xival{k}\} \to
  \Laugt(\xvb,\lmvb,\bar \nu, \bar \eta, \{\xival{k}\})$, such that
  $(\xvb, \lmvb, \bar \nu, \bar \eta, \{ (\bar \xi)^k\})$ is a saddle
  point of $\Laugt$, completing the proof.
\end{proof}

Section~\ref{sec:dist1} describes classes of objective functions for
which the hypotheses of Proposition~\ref{pr:saddle-primal-dual-equiv}
are met.  We have introduced in
Proposition~\ref{pr:saddle-primal-dual-equiv} the set $\CC$ to
increase the level of generality in preparation for the exposition of
our algorithm that follows next. Specifically, since $f$ is not
necessarily convex-concave, the function $\Laugt$ might not be
convex-concave over the entire domain $(\real^{nd} \times
\realnonnegative^n) \times (\real^n \times \real^{nd} \times
\real^{mN})$. For such cases, one can restrict the attention to the
set $\CC \times (\real^n \times \real^{nd} \times \real^{mN})$
provided the hypotheses of the above result are satisfied. As we show
later, when the objective function $f$ is convex-concave, one can
employ the set $\CC = \real^{nd} \times \realnonnegative^n$.

\section{Distributed algorithm design and convergence
analysis}\label{sec:dist}

Here we design and analyze our distributed algorithm to find the
solutions of the optimization
problem~\eqref{eq:stochastic-net-c-opt}. Given the results of
Section~\ref{sec:reformulation}, and specifically
Proposition~\ref{pr:saddle-primal-dual-equiv}, our algorithm seeks to
find the saddle points of $\Laugt$ over the domain $\CC \times
(\real^n \times \real^{nd} \times \real^{mN})$. The dynamics consists
of (projected) gradient-descent of $\Laugt$ in the convex variables
and gradient-ascent in the concave ones. This is popularly termed as
the saddle-point or the primal-dual
dynamics~\cite{AC-BG-JC:17-sicon,AC-EM-JC:16-scl}. 

Given a closed, convex set $\CC \subset \real^{nd} \times
\realnonnegative^n$, the saddle-point dynamics for $\Laugt$ is
\begin{subequations}\label{eq:dyn3}
  \begin{align}
    \begin{bmatrix} \frac{d \xvec}{dt} \\ \frac{d \lmvec}{dt}
    \end{bmatrix} 
    & = \Pi_\CC \Bigl( (\xvec,\lmvec), \begin{bmatrix} - \gradient_{\xvec}
    \Laugt(\xvec,\lmvec,\nu,\eta,\{\xival{k}\})  
    \\
    - \gradient_{\lmvec}
    \Laugt(\xvec,\lmvec,\nu,\eta,\{\xival{k}\}) \end{bmatrix} \Bigr),
    \label{eq:dyn3-2}
    \\
    \frac{d \nu}{dt}  & = \gradient_{\nu}
    \Laugt(\xvec,\lmvec,\nu,\eta,\{\xival{k}\}),\label{eq:dyn3-3}
    \\
    \frac{d \eta}{dt} & = \gradient_{\eta}
    \Laugt(\xvec,\lmvec,\nu,\eta,\{\xival{k}\}),\label{eq:dyn3-4}
    \\
    \frac{d \xival{k}}{dt} & = \gradient_{\xival{k}}
    \Laugt(\xvec,\lmvec,\nu,\eta,\{\xival{k}\}), \, \forall k \in
    \until{N}. \label{eq:dyn3-5}
  \end{align}
\end{subequations}
For convenience, denote~\eqref{eq:dyn3} by the vector field
$\map{\spLaugt}{\real^{nd} \times \realnonnegative^n \times
  \real^{nd+n+mN}}{\real^{nd} \times \realnonnegative^n \times
  \real^{nd+n+mN}}$. In this notation, the first, second, and third
components correspond to the dynamics of $\xvec$, $\lmvec$, and
$(\nu,\eta,\{\xival{k}\})$, respectively.

\begin{remark}\longthmtitle{Distributed implementation
    of~\eqref{eq:dyn3}}\label{re:dist-imp-new} 
  {\rm Here we discuss the distributed character of the
    dynamics~\eqref{eq:dyn3}. For this, we rely on the set $\CC$ being
    decomposable into constraints on individual agent's decision
    variables, i.e., $\CC := \Pi_{i=1}^n \CC_i$ with $\CC_i \subset
    \real^d \times \realnonnegative$.  This allows agents to perform
    the projection in~\eqref{eq:dyn3-2} in a distributed way (we show
    later that the set $\CC$ enjoys this structure for a broad class
    of objective functions~$f$).  Denote the components of the dual
    variables $\eta$ and $\nu$ by $\eta=(\etaval{1}; \etaval{2};
    \dots; \etaval{n})$ and $\nu=(\nuval{1}; \nuval{2}; \dots;
    \nuval{n})$, so that agent $i \in \until{n}$ maintains $\etaval{i}
    \in \real^d$ and $\nuval{i} \in \real$. Further, let $\KK_i
    \subset \until{N}$ be the set of indices representing the samples
    held by~$i$ ($k \in \KK_i$ if and only if $\data{k} \in
    \Xihat_i$). For implementing $\spLaugt$, we assume that each agent
    $i$ maintains and updates the variables
    $(\xval{i},\lmval{i},\nuval{i},\etaval{i},\{\xival{k}\}_{k \in
      \KK_i})$. The collection of these variables for all $i \in
    \until{n}$ forms $(\xvec,\lmvec,\nu,\eta,\{\xival{k}\})$.
    From~\eqref{eq:dyn3},
    the dynamics of variables maintained by~$i$ is
    \begin{align*}
      \Bigl(\frac{d \xval{i}}{dt}; \frac{d \lmval{i}}{dt} \Bigr) &=
      \Pi_{\CC_i} \Bigl( -\frac{1}{N} \sum_{k \in \KK_i}
      \gradient_{x} g_k(\xval{i},\lmval{i},\xival{k}) 
      \\
      & \qquad - \sum_{j \in \NN_i} a_{ij} \bigl((\etaval{i} -
      \etaval{j}) + (\xval{i} - \xval{j})\bigr); 
      \\
      & \qquad - \frac{\rads}{n} - 
      \frac{1}{N} \sum_{k \in \KK_i} \gradient_{\lm}
      g_k(\xval{i},\lmval{i},\xival{k})
      \\
      &  \qquad  - \sum_{j \in \NN_i} a_{ij} \bigl( (\nuval{i} -
    \nuval{j}) + (\lmval{i} - \lmval{j}) \bigr) \Bigr) ,
      \\
      \frac{d \nuval{i}}{dt} & =   \sum_{j \in \NN_i} a_{ij} (\lmval{i} -
      \lmval{j}),
      \\
      \frac{d \etaval{i}}{dt} & =  \sum_{j \in \NN_i} a_{ij} (\xval{i} -
      \xval{j}),
      \\
      \frac{d \xival{k}}{dt} & =  \frac{1}{N} \gradient_{\xi}
      g_k(\xval{i},\lmval{i},\xival{k}), \quad \forall k \in \KK_i.
  \end{align*}
  Observe that the right-hand side of the above dynamics is computable
  by agent $i$ using the variables that it maintains and information
  collected from its neighbors. 
    Hence, $\spLaugt$ can be implemented in
  a distributed manner. Note that the number of variables in the set
  $\{\xival{k}\}$, grows with the size of the data, whereas the size of
  all other variables is independent of the number of samples.
  Further, for any agent $i$, $\{\xival{k}\}_{k \in \KK_i}$ can be
interpreted as its internal state that is not communicated to its
neighbors.
\oprocend }
\end{remark}

The following result establishes the convergence of the dynamics
$\spLaugt$ to the saddle points of $\Laugt$. In our previous
works~\cite{AC-EM-JC:16-scl, AC-BG-JC:17-sicon,AC-EM-SHL-JC:18-tac},
we have extensively analyzed the convergence properties of
saddle-point dynamics associated to convex-concave functions.
However, those results do not apply directly to infer convergence
for~$\spLaugt$ because projection operators are involved in our
algorithm design, $\Laugt$ is linear in both convex
  ($\lmvec$) and concave ($\nu$, $\eta$) variables (which prevents it
  from being strictly convex-concave), and $\Laugt$ is not linear in
  the concave variable $\{\xival{k}\}$.
Nonetheless, we borrow much insight from our previous analysis to
prove the following result.

\begin{theorem}\longthmtitle{Convergence of trajectories of $\spLaugt$
    to the optimizers
    of~\eqref{eq:stochastic-net-dist-opt}}\label{th:convergence-dyn3}
  Suppose the hypotheses of
  Proposition~\ref{pr:saddle-primal-dual-equiv} hold. Assume further
  that there exists a saddle point
  $(\xvec^*,\lmvec^*,\nu^*,\eta^*,\{(\xival{k})^*\})$ of $\Laugt$
  with $(\xvec^*,\lmvec^*) \in \intr(\CC)$ 
  such that the map
  	$\xi \mapsto g_k( (\xvec^*)^{v_k}, (\lmvec^*)^{v_k} , \xi)$ 
  	is strongly concave for all $k \in \until{N}$.
  Then, the trajectories
  of~\eqref{eq:dyn3} starting in $\CC \times \real^n \times \real^{nd}
  \times \real^{mN}$ remain in this set and converge asymptotically to a
  saddle point of $\Laugt$. As a consequence, the $(\xvec,\lmvec)$
  component of the trajectory converges to an optimizer
  of~\eqref{eq:stochastic-net-dist-opt}.
\end{theorem}
\begin{proof}
  We understand the trajectories of~\eqref{eq:dyn3} in the
  Caratheodory sense, cf. Section~\ref{subsec:disc}. Note that by
  definition of the projection operator, any solution $t \mapsto
  (\xvec(t),\lmvec(t),\nu(t),\eta(t),\{\xival{k}(t)\})$
  of~\eqref{eq:dyn3} starting with $(\xvec(0),\lmvec(0)) \in \CC$
  satisfies $(\xvec(t),\lmvec(t)) \in \CC$ for all $t \ge 0$.

  \vspace*{1ex}
  \noindent \emph{LaSalle function.} Let
  $(\xvec^*,\lmvec^*,\nu^*,\eta^*,\{(\xi^*)^{k}\})$ be the equilibrium
  point of $\Laugt$ satisfying $(\xvec^*,\lmvec^*) \in \intr(\CC)$.
  Using the definition of equilibrium point in~\eqref{eq:dyn3-3}
  and~\eqref{eq:dyn3-4}, we get
  \begin{align}\label{eq:consensus-prop-2}
    (\Lap \otimes \identity{d}) \xvec^* = \zeros_{nd} \text{ and } \Lap
    \lmvec^* = \zeros_n .
  \end{align}
  Consider the function $\map{V}{\CC \times
    \real^{nd+n+Nm}}{\realnonnegative}$,
  \begin{align*}
    V(\xvec,\lmvec,\zeta) := \frac{1}{2}(\norm{\xvec - \xvec^*}^2 +
    \norm{\lmvec - \lmvec^*}^2 + \norm{\zeta - \zeta^*}^2),
  \end{align*}
  where, for convenience, we use $\zeta:=(\nu,\eta,\{\xival{k}\})$ and,
  likewise, $\zeta^*:=(\nu^*,\eta^*,\{(\xi^*)^{k}\})$.  
  Writing the dynamics~\eqref{eq:dyn3} as $(-\gradient_{\xvec} \Laugt ; -\gradient_{\lmvec} \Laugt; \gradient_{\zeta} \Laugt) - (\varphi_{\xvec}; \varphi_{\lmvec}; \zeros_{nd+n+Nm})$ where $(\varphi_{\xvec},\varphi_{\lmvec})$ is an element of the normal
  cone $N_{\CC}(\xvec,\lmvec)$ (cf. Section~\ref{subsec:convex-a}) and following the steps of~\cite[Proof of Lemma 4.1]{AC-EM-JC:16-scl}, we obtain that the Lie derivative of $V$ along the dynamics~\eqref{eq:dyn3} satisfies the bound
  \begin{align*}
	 \Lie_{\spLaugt} V(\xvec,\lmvec,\zeta)\le & \Laugt(\xvec^*,\lmvec^*,\zeta) -
	\Laugt(\xvec^*,\lmvec^*,\zeta^*)
	\\
	& \quad + \Laugt(\xvec^*,\lmvec^*,\zeta^*) -
	\Laugt(\xvec,\lmvec,\zeta^*) .
  \end{align*}
  From the definition of saddle point, the sum of the first two terms of
  the right-hand side are nonpositive and so is the sum of the last two.
  Therefore, we conclude  
  \begin{align}\label{eq:auxx}
    \Lie_{\spLaugt} V(\xvec,\lmvec,\zeta) \le 0 .
  \end{align}

  \vspace*{1ex}
  \noindent \emph{Application of LaSalle invariance principle.}  Using
  the property~\eqref{eq:auxx}, we deduce two facts.
  First, given $\delta \ge 0$, any trajectory of~\eqref{eq:dyn3}
  starting in $\SS_{\delta} := V^{-1}(\le \delta) \cap (\CC \times
  \real^{n+nd+mN})$ remains in $\SS_\delta$ at all times. In
  particular, every equilibrium point is stable under the
  dynamics. Second, the omega-limit set of each trajectory
  of~\eqref{eq:dyn3} starting in $\SS_\delta$ is invariant under the
  dynamics.  Thus, from the invariance principle for discontinuous
  dynamical systems, cf. Proposition~\ref{pr:invariance-cara}, any
  solution of~\eqref{eq:dyn3} converges to the largest invariant set
  \begin{align*}
    \MM \subset
  \setdef{(\xvec,\lmvec,\zeta)}{\Lie_{\spLaugt} V(\xvec,\lmvec,\zeta)
    = 0, (\xvec,\lmvec) \in \CC}.
  \end{align*}

  \vspace*{1ex}
  \noindent \emph{Properties of the largest invariant set.}  Let
  $(\xvec,\lmvec,\zeta) \in \MM$. Then, from $\Lie_{\spLaugt}
  V(\xvec,\lmvec,\zeta) = 0$ and our bounding above, we get
  \begin{align}
    \Laugt(\xvec^*,\lmvec^*,\zeta) \!\overset{(a)}{=} \!
    \Laugt(\xvec^*,\lmvec^*,\zeta^*) \! \overset{(b)}{=} \!
    \Laugt(\xvec,\lmvec,\zeta^*). \label{eq:two-eq-2}
  \end{align} 
  Expanding the equality $(a)$ and using~\eqref{eq:consensus-prop-2}, we
  obtain
  \begin{align}
      &  \textstyle  \sum_{k=1}^N g_k( (\xvec^*)^{v_k},
    (\lmvec^*)^{v_k} ,  \xival{k}) \notag
    \\ 
    & \textstyle \qquad \qquad  =  \sum_{k=1}^N g_k( (\xvec^*)^{v_k},
    (\lmvec^*)^{v_k}, (\xi^*)^{k}) . \label{eq:g-eq-new}
  \end{align}
  From the saddle-point property,
  $\{(\xi^*)^{k}\}$ maximizes the function $\{\xival{k}\} \mapsto
  \sum_{k=1}^N g_k((x^*)^{v_k}, (\lm^*)^{v_k} , \xival{k})$. This map
  is strongly concave by hypothesis.  
  Therefore,~\eqref{eq:g-eq-new} yields
    $\xival{k} = (\xi^*)^{k}, \text{ for all } k \in \until{N}$.
  Expanding the equality $(b)$ in~\eqref{eq:two-eq-2} and
  using~\eqref{eq:consensus-prop-2}, we get
  \begin{align}
     &  \textstyle h(\lmvec^*) + \frac{1}{N} \sum_{k=1}^N g_k( \xvec^*)^{v_k},
    (\lmvec^*)^{v_k}, (\xi^*)^{k}) = h(\lmvec) \notag
    \\
    & \textstyle \quad +  \frac{1}{N} \sum_{k=1}^N g_k(\xvec^{v_k},
    \lmvec^{v_k},(\xi^*)^{k}) + (\nu^*)^\top \Lap \lmvec \notag 
    \\
    &  \quad + (\eta^*)^\top (\Lap \otimes \identity{d}) \xvec + \frac{1}{2}
    \xvec^\top (\Lap \otimes \identity{d}) \xvec + \frac{1}{2}
    \lmvec^\top \Lap \lmvec  . \label{eq:equality-b-new}
  \end{align}
  For ease of notation, let $\yvec := (\xvec;\lmvec)$, $\yvec^* :=
  (\xvec^*; \lmvec^*)$, and
  \begin{align*}
    G(\yvec) := h(\lmvec) +  \frac{1}{N} \sum_{k=1}^N g_k( \xvec^{v_k},
    \lmvec^{v_k},(\xi^*)^{k}).
  \end{align*}
  Then, the expression~\eqref{eq:equality-b-new} can be written as
  \begin{align}
    G(\yvec^*) & = G(\yvec) + (\nu^*)^\top \Lap \lmvec + (\eta^*)^\top
    (\Lap \otimes \identity{d}) \xvec \notag
    \\
    & \quad + \frac{1}{2} \yvec^\top (\Lap \otimes \identity{d+1})
    \yvec. \label{eq:equality-b-2-new}
  \end{align}
  From the definition of saddle point, $(\xvec^*,\lmvec^*)$ minimizes
  the function $(\xvec,\lmvec) \mapsto \Laugt(\xvec,\lmvec,\zeta^*)$
  over the domain $\CC$. Moreover, by assumption $(\xvec^*,\lmvec^*)$
  lies in the interior of $\CC$.  Thus,
  \begin{subequations}\label{eq:eq-cond-xl}
    \begin{align}
      \gradient_{\xvec} \Laugt(\xvec^*,\lmvec^*,\zeta^*) & =
      \zeros_{nd},\label{eq:eq-cond-xl-1}
      \\
      \gradient_{\lmvec} \Laugt(\xvec^*,\lmvec^*,\zeta^*) & =
      \zeros_{n}.\label{eq:eq-cond-xl-2}
    \end{align}
  \end{subequations}
  The first of the above equalities yield 
    $(\Lap \otimes \identity{d}) \eta^* = - \gradient_{\xvec}
    G(\yvec^*)$.
  Plugging this equality in~\eqref{eq:equality-b-2-new} and rearranging
  terms gives
  \begin{align}
    \frac{1}{2} \yvec^\top (\Lap \otimes \identity{d+1}) \yvec  & =
    G(\yvec^*) - G(\yvec) \notag
    \\
    & - (\nu^*)^\top \Lap \lmvec  + \xvec^{\top} \gradient_{\xvec}
    G(\yvec^*). \label{eq:equality-b-3-new}
  \end{align}
  Note that $(\xvec^*)^\top \gradient_{\xvec} G(\yvec^*) =
  (\xvec^*)^\top \bigl( \gradient_{\xvec} G(\yvec^*) + (\Lap \otimes
  \identity{d}) \eta^* + (\Lap \otimes \identity{d}) \xvec^* \bigr) =
  (\xvec^*)^\top \gradient_{\xvec} \Laugt(\xvec^*,\lmvec^*,\zeta^*)$,
  where we have used~\eqref{eq:consensus-prop-2}. This in turn equals
  $0$ because of~\eqref{eq:eq-cond-xl-1}.
  Thus, we can rewrite~\eqref{eq:equality-b-3-new} as
  \begin{align}
    \frac{1}{2} \yvec^\top & (\Lap \otimes \identity{d+1}) \yvec  =
    G(\yvec^*) - G(\yvec) \notag
    \\
    & \quad - (\nu^*)^\top \Lap \lmvec + (\xvec - \xvec^*)^\top \gradient_{\xvec}
    G(\yvec^*) \label{eq:equality-b-4-new}
  \end{align}
  Expanding~\eqref{eq:eq-cond-xl-2} gives
  \begin{align}
    \gradient_{\lmvec} G(\yvec^*) + \Lap \nu^* + \frac{1}{2} \Lap
    \lmvec^* = 0. \label{eq:F-grad-zero}
  \end{align}
  Pre-multiplying the above equation with $(\lmvec^*)^\top$ and
  using~\eqref{eq:consensus-prop-2}, we get $(\lmvec^*)^\top \gradient_{\lmvec}
  G(\yvec^*) = 0$ and we can further
  rewrite~\eqref{eq:equality-b-4-new} as
  \begin{align}
    \frac{1}{2} \yvec^\top & (\Lap \otimes \identity{d+1}) \yvec  =
    G(\yvec^*) - G(\yvec)- (\nu^*)^\top \Lap \lmvec  \notag
    \\
     & + (\xvec - \xvec^*)^\top \gradient_{\xvec}
     G(\yvec^*) - (\lmvec^*)^\top \gradient_{\lmvec} G(\yvec^*)
     \label{eq:equality-b-5-new}
  \end{align}
  Using~\eqref{eq:consensus-prop-2} in~\eqref{eq:F-grad-zero} yields
  $\gradient_{\lmvec} G(\yvec^*) = - \Lap \nu^*$. That is, $\lmvec^\top
  \gradient_{\lmvec} G(\yvec^*) = - \lmvec^\top \Lap \nu^*$ which then
  replaced in~\eqref{eq:equality-b-5-new} gives 
  \begin{align*}
    \frac{1}{2} \yvec^\top & (\Lap \otimes \identity{d+1}) \yvec  =
    G(\yvec^*) - G(\yvec) 
    + (\yvec - \yvec^*)^\top \gradient_{\yvec} G(\yvec^*). 
  \end{align*}
  The first-order convexity condition for $F$ takes the form
  \begin{align*}
    G(\yvec) \ge G(\yvec^*) + (\yvec - \yvec^*)^\top \gradient_{\yvec}
    G(\yvec^*) 
  \end{align*}
  Using the previous two expressions, we obtain 
    $\yvec^\top (\Lap \otimes \identity{d+1}) \yvec \le 0$.
  This is only possible if this expression is zero because $\Lap \otimes
  \identity{d+1}$ is positive semidefinite. Equating it to zero, we get
  $\xvec = \ones_n \otimes x$ and $\lmvec = \lm \ones_n$ for some
  $(x,\lm)$ and $(\xvec,\lmvec) \in \CC$.  Collecting our derivations so
  far, we have that if $(\xvec,\lmvec,\zeta) \in \MM$, then
  \begin{subequations}\label{eq:conditions-inv-set-new}
    \begin{align}
      \xival{k}   = (\xi^*)^{k}, & \, \forall k \in \until{N},
      \qquad 
      \xvec  = \ones_n \otimes x, 
      \\
      & \lmvec  = \lm \ones_n, \, \, \, (\xvec,\lmvec) \in \CC.
    \end{align}
  \end{subequations}

  \vspace*{1ex}
  \noindent \emph{Identification of the largest invariant set.}
  Consider a trajectory $t \mapsto (\xvec(t),\lmvec(t),\zeta(t))$
  of~\eqref{eq:dyn3} starting at $(\xvec(0),\lmvec(0),\zeta(0)) \in
  \MM$ and remaining in $\MM$ at all times (recall that $\MM$ is
  invariant).  Then, the trajectory must
  satisfy~\eqref{eq:conditions-inv-set-new} for all $t \ge 0$, that
  is, there exists $t \mapsto (x(t),\lm(t))$ such that
   \begin{subequations}\label{eq:conditions-inv-set-t-new}
    \begin{align}
      \xival{k}(t)  = (\xi^*)^{k},  \, \forall k \in  \until{N},
      \qquad 
      \xvec(t)  = \ones_n \otimes x(t), 
      \\
      \lmvec(t)  = \lm(t) \ones_n, \, \, \, (\xvec(t),\lmvec(t)) \in \CC,
    \end{align}
  \end{subequations}
  for all $t \ge 0$. Plugging~\eqref{eq:conditions-inv-set-t-new}
  in~\eqref{eq:dyn3}, we obtain that for all $t \ge 0$, along the
  considered trajectory, we have  $\dot \nu(t) = \zeros_n$, $\dot
  \eta(t) = \zeros_{nd}$, and $\dot \xi(t) = \zeros_{mN}$. This implies
  that the considered trajectory satisfies the following for all $t \ge
  0$,
  \begin{align*}
    \begin{bmatrix} \frac{d \xvec(t)}{dt} \\ \frac{d \lmvec(t)}{dt}
    \end{bmatrix} \! \! \! = \! \Pi_\CC \Bigl( \! (\xvec(t),\lmvec(t)),
    \! \! 
    \begin{bmatrix} \! - \! \gradient_{\xvec}
      \Laugt(\xvec(t),\lmvec(t),\zeta(0)) 
    \\
    \! - \! \gradient_{\lmvec} \Laugt(\xvec(t),\lmvec(t),\zeta(0))
    \end{bmatrix} \Bigr)
  \end{align*}
  which is a gradient descent dynamics of the convex function
  $(\xvec,\lmvec) \mapsto \Laugt(\xvec,\lmvec, \zeta(0))$ projected over the
  set $\CC$. Thus, either $t \mapsto \Laugt(\xvec(t),
  \lmvec(t),\zeta(0))$ decreases at some $t$ or the right-hand side of
  the above dynamics is zero at all times. Note that for all $t \ge 0$,
  \begin{align*}
    \Laugt(\xvec(t),\lmvec(t),\zeta(0)) \overset{(a)}{=} \Laugt(\ones_n
    \otimes x(t), \lm(t) \ones_n, \zeta(0))
    \\
    \overset{(b)}{=} h(\lm(t) \ones_n) + \frac{1}{N} \sum_{k=1}^N
    g_k(\ones_n \otimes x(t), \lm(t) \ones_n, (\xi^*)^k)
    \\
    \overset{(c)}{=} \Laugt(\ones_n \otimes x(t), \lm(t) \ones_n,
    \zeta^* ) \overset{(d)}{=} \Laugt(\xvec^*,\lmvec^*,\zeta^*).
  \end{align*}
  In the above set of expressions, equalities (a), (b), and (c) follow
  from conditions~\eqref{eq:conditions-inv-set-t-new} and the
  definition of $\Laugt$. Equality~(d) follows
  from~\eqref{eq:two-eq-2}, which holds from every point in~$\MM$. The
  above implies that $t \mapsto \Laugt(\xvec(t),\lmvec(t),\zeta(0))$
  is a constant map.  As a consequence, we conclude that
  $(\xvec(0),\lmvec(0),\zeta(0))$ is an equilibrium point
  of~\eqref{eq:dyn3}. Therefore, we have proved that the set $\MM$ is
  entirely composed of the equilibrium points of the
  dynamics~\eqref{eq:dyn3}.  Convergence to an equilibrium point in
  the set of saddle points for each trajectory follows from this and
  the fact that each equilibrium point is stable, cf.~\cite{SPB-DSB:03}.
\end{proof}

\begin{remark}\longthmtitle{Convergence of algorithm for  nonsmooth
    objective functions}\label{re:nonsmooth-f}
  {\rm Let $f$ satisfy all assumptions outlined in
    Section~\ref{sec:problem} except the differentiability and instead
    assume it is locally Lipschitz. This implies that the gradient of
    $\Laugt$ with respect to variables $\xvec$ and $\{\xival{k}\}$
    need not exist everywhere.  However, the generalized gradients
    exist, see e.g.,~\cite{JC:08-csm-yo} for the
    definition. Therefore, one can replace gradients
    in~\eqref{eq:dyn3-2} and~\eqref{eq:dyn3-5} with the generalized
    counterparts and end up with a differential inclusion for the
    $\{\xival{k}\}$ dynamics and a projected differential inclusion
    for the $\xvec$ dynamics. Although we do not explore it here, we
    believe that, using analysis tools of nonsmooth dynamical systems,
    see~\cite{JC:08-csm-yo} and references therein, one can show that
    the trajectories of the resulting nonsmooth dynamical system
    retain the convergence properties of
    Theorem~\ref{th:convergence-dyn3}. A promising route to establish
    this is to follow the exposition of~\cite{RG:17-scl}, which
    studies saddle-point dynamics for a general class of functions.
    \oprocend }
\end{remark}

  \begin{remark}\longthmtitle{Discrete-time primal-dual
      algorithms}\label{re:discrete}
    {\rm Note that the practical implementation of the saddle-point
      dynamics requires a careful analysis of aspects such as the
      discretization scheme, communication efficiency, and robustness
      to asynchronous updates and packet drops. The Lyapunov-based
      perspective taken here provides an appealing approach to deal
      with these challenges. For example, triggered implementations
      result in communication-efficient discretization schemes, see
      e.g.,~\cite{DR-JC:16-sicon,SSK-JC-SM:15-auto} and input-to-state
      stability provides convergence guarantees under noisy updates,
      see e.g.~\cite{DMN-JC:16-sicon}. These facts motivate us to
      write the algorithm in continuous time and analyze it using
      Lyapunov/LaSalle arguments. Numerous other works analyze the
      saddle-point dynamics in discrete time, see
      e.g.,~\cite{AN-AO:09-jota} for a general setup
      and~\cite{DMN-JC:17-tac} for a distributed implementation.  }
    \oprocend
\end{remark}

\begin{remark}\longthmtitle{Constrained stochastic
    optimization}\label{re:cons-opt}
  {\rm Certain constrained stochastic optimization problems can be
    cast in the form~\eqref{eq:stochastic-net-opt-0} and are therefore
    amenable to the distributed algorithmic solution techniques
    developed here.  Given $\delta \in (0,1)$ and a measurable map
    $\map{F}{\real^n \times \real^m}{\real}$, consider the following
    constrained stochastic optimization problem
    \begin{align}\label{eq:cons-stochastic}
      %
      \underset{x \in \setdef{\real^d}{ \Pb(F(x,\xi) \le 0 ) \ge
          1-\delta}}{\inf} \Eb_{\Pb} \bigl[f(x,\xi)\bigr].
    \end{align}
    The constraint is probabilistic in nature and so is commonly
    referred to as \emph{chance constraint}~\cite{AS-DD-AR:14}.
    One approach to solve this problem is to remove the constraint and add
    a convex function to the objective that penalizes its
    violation. Conditional value-at-risk $(\cvar)$ is one such
    penalizing function. Formally, the $\cvar$ of $\xi \mapsto
    F(x,\xi)$ at level $\delta$ is
    \begin{align*}
      \cvar_\delta(F(x,\xi)) := \inf_{\tau \in \real} \Eb_\Pb \Bigl[
      \tau + \frac{1}{\delta} \max\{F(x,\xi) - \tau,0\} \Bigr].
    \end{align*}
    Roughly speaking, this value represents the expectation of $\xi
    \mapsto F(x,\xi)$ over the set of $\xi$ that has measure $\delta$
    and that contain the highest values of this function.  Note the
    fact~\cite[Chapter 6]{AS-DD-AR:14} that $\cvar_\delta(F(x,\xi))
    \le 0$ implies $\Pb(F(x,\xi) \le 0) \ge 1 - \delta$.  Thus, using
    $\cvar$, problem~\eqref{eq:cons-stochastic} can be approximated by
    \begin{align*}
      \inf_{x \in \real^d} \Eb_{\Pb}\bigl[f(x,\xi)\bigr] + \rho \, \cvar_\delta
      (F(x,\xi)),
    \end{align*}
    where $\rho > 0$ determines the trade-off between the two goals:
    minimizing the objective and satisfying the constraint. By
    invoking the definition of $\cvar$, the above problem can be
    written compactly as
    \begin{align*}
      \inf_{x \in \real^d, \tau \in \real} \Eb_{\Pb}
      \Bigl[ f(x,\xi) + \rho( \tau + \frac{1}{\delta}
      \max \{F(x,\xi)-\tau, 0 \} ) \Bigr].
    \end{align*}
    This can be further recast as a stochastic optimization of the
    form~\eqref{eq:stochastic-net-opt-0}. Therefore, under appropriate
    conditions on the function $F$, one can solve a chance-constrained
    problem in a distributed way under the data-driven optimization
    paradigm using the algorithm design introduced here.  \oprocend }
\end{remark}

\section{Objective functions that meet the algorithm convergence
  criteria}\label{sec:dist1}

In this section we report on two broad classes of objective
functions~$f$ for which the hypotheses of
Proposition~\ref{pr:saddle-primal-dual-equiv} hold. For both cases, we
justify how the dynamics~\eqref{eq:dyn3} serves as the distributed
algorithm for solving~\eqref{eq:stochastic-net-dist-opt}.

\subsection{Convex-concave functions}

Here we focus on objective functions that are convex-concave in
$(x,\xi)$. That is, in addition to $x \mapsto f(x,\xi)$ being convex
for each $\xi \in \real^m$, the function $\xi \mapsto f(x,\xi)$ is
concave for each $x \in \real^d$. We proceed to check the hypotheses
of Theorem~\ref{th:convergence-dyn3}. To this end, let $\CC =
\real^{nd} \times \realnonnegative^n$, which is closed, convex set
with $\intr(\CC) \not = \emptyset$.  Note that $\Laugt$ is
convex-concave on $\CC \times (\real^n \times \real^{nd} \times
\real^{mN})$ as $f$ is convex-concave.  The following result shows
that~\eqref{eq:min-max-exchange} holds.

\begin{lemma}\longthmtitle{Min-max operators can be interchanged
    for~$\Laugt$}\label{le:interchange} 
  Let $f$ be convex-concave in $(x,\xi)$.  Then, for any $(\nu,\eta)
  \in \real^n \times \real^{nd}$, the following holds
  \begin{align}
    &\min_{\xvec,\lmvec \ge \zeros_n} \max_{\{\xival{k}\}}
    \Laugt(\xvec,\lmvec,\nu,\eta,\{\xival{k}\}) \notag
    \\
    & \quad = \max_{\{\xival{k}\}} \min_{\xvec,\lmvec \ge \zeros_n}
    \Laugt(\xvec,\lmvec,\nu,\eta,\{\xival{k}\}).
    \label{eq:minmax-prop-lb}
  \end{align}
\end{lemma}
\begin{proof}
  Given any $(\nu,\eta)$, denote the function
  $(\xvec,\lmvec,\{\xival{k}\}) \mapsto
  \Laugt(\xvec,\lmvec,\nu,\eta,\{\xival{k}\})$ by
  $\Laugt^{(\nu,\eta)}$.  Since $f$ is convex-concave, so is
  $\Laugt^{(\nu,\eta)}$ in the variables $( (\xvec,\lmvec),
  \{\xival{k} \})$.  We use Theorem~\ref{th:saddle-value} to prove the
  result.  To do so, let us extend $\Laugt^{(\nu,\eta)}$ over the
  entire domain $(\real^{nd} \times \real^n) \times (\real^{mN})$ as
  \begin{align*}
    \Laugr^{(\nu,\eta)}(\xvec,\lmvec,\{\xival{k}\}) = \begin{cases}
      \Laugt^{(\nu,\eta)}(\xvec,\lmvec,\{\xival{k}\}), & \, \text{ if
      } \lmvec \ge \zeros_n,
      \\
      +\infty, & \, \text{ otherwise}.
    \end{cases}
  \end{align*}
  One can see that $\Laugr^{(\nu,\eta)}$ is closed, proper, and
  convex-concave (cf. Section~\ref{sec:prelims} for definitions).
  Further, following~\cite[Theorem 36.3]{RTR:97}, the
  equality~\eqref{eq:minmax-prop-lb} holds if and only if the
  following holds
  \begin{align*}
    \min_{\xvec,\lmvec} \max_{\{\xival{k}\}} & \,
    \Laugr^{(\nu,\eta)}(\xvec,\lmvec,\{\xival{k}\})
    \\
    & \qquad = \max_{\{\xival{k}\}} \min_{\xvec,\lmvec}
    \Laugr^{(\nu,\eta)}(\xvec,\lmvec,\{\xival{k}\}).
  \end{align*}
  The rest of the proof establishes the above condition by checking
  the hypotheses of Theorem~\ref{th:saddle-value} for
  $\Laugr^{(\nu,\eta)}$. 
  For showing Theorem~\ref{th:saddle-value}(i), it is enough to
  identify $\{\xibval{k}\} \in \real^{mN}$ for which the function
  $(\xvec,\lmvec) \mapsto \Laugr^{(\nu,\eta)}(\xvec, \lmvec,
  \{\xibval{k}\})$ does not have a direction of recession.  By the
  assumptions on $f$, for each $k \in \until{N}$, there exists
  $\xibval{k} \in B_{\eps_N(\beta)\sqrt{N}/\sqrt{2n}}(\data{k})$ such
  that $\ones_n$ and $-\ones_n$ are not directions of recession for
  the function $x \mapsto f(x,\xibval{k})$. Picking these values, one
  has $\norm{\xibval{k} - \data{k}}^2 \le \rads N/2n$ for all
  $k \in \until{N}$. Thus,
  \begin{align*}
    \Laugr^{(\nu,\eta)}(\xvec,\lmvec,\{\xibval{k}\}) = \frac{\rads (z^t \lmvec)}{n} + \frac{1}{N} \sum_{k=1}^N
    f(\xval{v_k}, \xibval{k} )
    \\
    + \nu^\top \Lap \lmvec + \eta^\top (\Lap \otimes \identity{d})
    \xvec + \frac{1}{2} \xvec^\top (\Lap \otimes \identity{d}) \xvec +
    \frac{1}{2} \lmvec^\top \Lap \lmvec,
  \end{align*}
  where $z \in \real^n$ with $z_i > 0$ for all $i \in \until{n}$. 
  One can show that the right-hand side of the above expression as a
  function of $(\xvec,\lmvec)$ does not have a direction of recession,
  that is, Theorem~\ref{th:saddle-value}(i) holds.  Next, we check
  Theorem~\ref{th:saddle-value}(ii). We show that there exists
  $(\xvecb,\lmvecb) \in \ri(\real^{nd} \times \realnonnegative^n)$
  such that the function $\{\xival{k}\} \mapsto
  -\Laugr^{(\nu,\eta)}(\xvecb,\lmvecb,\{\xival{k}\})$ does not have a
  direction of recession. To this end, pick $\xvecb = \ones_{nd}$ and
  $\lmvecb = \ones_n$. Then,
  \begin{align*} \Laugr^{(\nu,\eta)}
    (\xvecb,\lmvecb,\{\xival{k}\}) \! = \! \rads \! + \!
    \frac{1}{N} \sum_{k=1}^N f(\ones_d,\xival{k}) \! - \!
    \norm{\xival{k} - \data{k}}^2.  
  \end{align*} 
  Recall that for any $x \in \real^d$, $\xi \mapsto f(x,\xi)$ is
  concave. Hence, we deduce from the above expression that
  $\Laugr^{(\nu,\eta)}(\xvecb,\lmvecb,\{\xival{k}\}) \to - \infty$ as
  $\norm{\{\xival{k}\}} \to \infty$. Therefore, $\{\xival{k}\} \mapsto
  -\Laugr^{(\nu,\eta)}(\xvecb,\lmvecb,\{\xival{k}\})$ does have a
  direction of recession, completing the proof.
\end{proof}

As a consequence of the above discussion, we conclude that the
hypotheses of Proposition~\ref{pr:saddle-primal-dual-equiv} hold true
for the considered class of objective functions, and we can state,
invoking Theorem~\ref{th:convergence-dyn3}, the following convergence
result.

\begin{corollary}\longthmtitle{Convergence of trajectories of
    $\spLaugt$ for convex-concave
    $f$}\label{cr:convergence-convex-concave}
  Let $f$ be convex-concave in $(x,\xi)$ and $\CC = \real^{nd} \times
  \realnonnegative$. Assume further that there exists a saddle point
  $(\xvec^*,\lmvec^*,\nu^*,\eta^*,\{(\xival{k})^*\})$ of $\Laugt$
  satisfying $\lmvec^* > \zeros_n$. Then, the trajectories
  of~\eqref{eq:dyn3} starting in $\CC \times \real^n \times \real^{nd}
  \times \real^{mN}$ remain in this set and converge asymptotically to
  a saddle point of $\Laugt$. As a consequence, the $(\xvec,\lmvec)$
  component of the trajectory converges to an optimizer
  of~\eqref{eq:stochastic-net-dist-opt}.
\end{corollary}

It is important to note that $\CC = \Pi_{i=1} (\real^d \times
\realnonnegative)$ and thus the projection in~\eqref{eq:dyn3-2} can be
executed by individual agents. Following Remark~\ref{re:dist-imp-new},
the dynamics~\eqref{eq:dyn3} is implementable in a distributed way.

\subsection{Convex-convex function}

Here we focus on objective functions for which both $x \mapsto
f(x,\xi)$ and $\xi \mapsto f(x,\xi)$ are convex maps for all $x \in
\real^d$ and $\xi \in \real^m$. Note that $f$ need not be jointly
convex in $x$ and $\xi$. We further divide this classification into
two.

\subsubsection{Quadratic function in $\xi$}\label{se:quadratic}

Assume additionally that the function $f$ is of the form
\begin{align}
  f(x,\xi) := \xi^\top Q \xi + x^\top R \xi + \ell(x), \label{eq:quad-f}
\end{align}
where $Q \in \real^{m \times m}$ is positive definite, $R \in \real^{d
  \times m}$, and $\ell$ is a continuously differentiable convex
function.  Our next result is useful in identifying a domain that
contains the saddle points of $\Laug$ over $(\real^{nd} \times
\realnonnegative^n) \times (\real^n \times \real^{nd})$.

\begin{lemma}\longthmtitle{Characterizing where the
    objective function  of~\eqref{eq:stochastic-net-dist-opt} is
    finite-valued}\label{le:finite-F-quad} 
  Assume $f$ is of the form~\eqref{eq:quad-f}. Then, the function $F$
  defined in~\eqref{eq:F} is finite-valued only if $\lmval{i}
  \ge \lm_{\max}(Q)$ for all $i \in \until{n}$.
\end{lemma}
\begin{proof}
  Assume there exists $\tilde i \in \until{n}$ such that
  $\lmval{\tilde i} < \lm_{\max} (Q)$. We wish to show that
  $F(\xvec,\lmvec) = +\infty$ in this case. For any $k$ such that
  $\data{k} \in \Xihat_{\tilde i}$, we have
  \begin{align*}
    g_k(\xval{\tilde i}, \lmval{\tilde i}, \xi) & = \xi^\top (Q -
    \lmval{\tilde i} \identity{m}) \xi + (\xval{\tilde i})^\top R \xi +
    2 \lmval{\tilde i} (\data{k})^\top \xi 
    \\
    & \quad +
    \ell(\xval{\tilde i}) - \lmval{\tilde i} \norm{\data{k}}^2.
  \end{align*}
  Let $w_{\max}(Q) \in \real^m$ be an eigenvector of $Q$ corresponding to the
  eigenvalue $\lm_{\max}(Q)$. Parametrizing $\xi = \alpha w_{\max}(Q)$,
  we obtain
  \begin{align*}
    g_k & (\xval{\tilde i}, \lmval{\tilde i}, \alpha w_{\max} (Q))  =
    \alpha^2 (\lm_{max}(Q) - \lmval{\tilde i}) \norm{w_{\max}(Q)}^2 
    \\
    & + \alpha \Bigl( (\xval{\tilde i})^\top R
    + 2 \lmval{\tilde i} (\data{k})^\top \Bigr)  w_{\max}(Q) 
    + \ell(\xval{\tilde i}) - \lmval{\tilde i} \norm{\data{k}}^2.
  \end{align*}
  Thus, we get $\max_{\alpha} g_k(\xval{\tilde i}, \lmval{\tilde i},
  \alpha w_{\max} (Q)) = + \infty$ and so $\max_{\xi} g_k(\xval{\tilde
  i}, \lmval{\tilde i}, \xi) = +\infty$. Further note that for any $i$
  and $k$ with $\data{k} \in \Xihat_i$, $\max_{\xi}
  g_k(\xval{i},\lmval{i},\xi) > - \infty$. This implies that
  $\sum_{k=1}^N \max_{\xi} g_k(\xval{v_k}, \lmval{v_k},\xi) = + \infty$
  and so $F(\xvec,\lmvec) = +\infty$.
\end{proof}

The above result implies that the optimizers
of~\eqref{eq:stochastic-net-dist-opt} for objective functions of the
form~\eqref{eq:quad-f} belong to the domain
\begin{align}
  \CC := \real^{nd} \times \setdef{\lmvec \in
  \realnonnegative^n}{\lmvec \ge \lm_{\max}(Q) \ones_n}. \label{eq:C-quad-f}
\end{align}
Therefore, the saddle points of $\Laug$ over the domain $(\real^{nd}
\times \realnonnegative^n) \times (\real^n \times \real^{nd})$ are
contained in the set $\CC \times (\real^n \times \real^{nd})$. Note
that $\CC$ is closed, convex with a nonempty interior. Furthermore,
following the proof of Lemma~\ref{le:finite-F-quad}, one can show that
$\Laugt$ is convex-concave on $\CC \times (\real^n \times \real^{nd}
\times \real^{mN})$. An easy way to validate this fact is by noting
that the Hessian of $\Laugt$ with respect to the convex (concave)
variables is positive (negative) semidefinite. Finally, repeating the
proof of Lemma~\ref{le:interchange}, we arrive at the
equality~\eqref{eq:min-max-exchange}. Using these facts in
Theorem~\ref{th:convergence-dyn3} yields the following result.

\begin{corollary}\longthmtitle{Convergence of trajectories of $\spLaugt$
for quadratic $f$}\label{cr:convergence-quadratic}
  Let $f$ be of the form~\eqref{eq:quad-f} and $\CC$ be given
  in~\eqref{eq:C-quad-f}. 
  Assume further that there exists a saddle point
  $(\xvec^*,\lmvec^*,\nu^*,\eta^*,\{(\xival{k})^*\})$ of $\Laugt$
  satisfying $\lmvec^* > \lm_{\max}(Q) \ones_n$. Then, the trajectories
  of~\eqref{eq:dyn3} starting in $\CC \times \real^n \times \real^{nd}
  \times \real^{mN}$ remain in this set and converge asymptotically to a
  saddle point of $\Laugt$. As a consequence, the $(\xvec,\lmvec)$
  component of the trajectory converges to an optimizer
  of~\eqref{eq:stochastic-net-dist-opt}.
\end{corollary}

Note that $\CC$ given in~\eqref{eq:C-quad-f} can be written as $\CC =
\Pi_{i=1}^n (\real^d \times \setdef{\lm \in \realnonnegative}{\lm \ge
\lm_{\max}(Q)})$. Thus, following Remark~\ref{re:dist-imp-new}, the
dynamics~\eqref{eq:dyn3} for this case can be
implemented in a distributed manner. 

\subsubsection{Least-squares problem}\label{se:least-squares}
Let $d = m$ and assume additionally that the function $f$ is of the
form
\begin{align}
  f(x,\xi) := a (\xi_m - (\xi_{1:m-1}; 1)^\top x)^2,
  \label{eq:f-least-sq} 
\end{align} 
where $a > 0$ and $\xi_{1:m-1}$ denotes the vector $\xi$ without the
last component $\xi_m$. Note that $f$ corresponds to the objective
function for a least-squares problem.  Further, note that it cannot be
written in the form~\eqref{eq:quad-f}, as can be seen from its
equivalent expression
\begin{align*}
  f(x,\xi) & = a \Bigl( \xi^\top (-x_{1:m-1};1) (-x_{1:m-1};1)^\top
  \xi
  \\
  & \qquad \qquad \qquad - 2 x_m (-x_{1:m-1};1)^\top \xi + x_m^2
  \Bigr).
\end{align*}
Our first step is to characterize, similarly to
Lemma~\ref{le:finite-F-quad}, the set over which the objective
function~\eqref{eq:F} takes finite values.

\begin{lemma}\longthmtitle{Characterizing where the objective function
    of~\eqref{eq:stochastic-net-dist-opt} is
    finite-valued}\label{le:finite-F-least-sq} 
  Assume $f$ is of the form~\eqref{eq:f-least-sq}. Then, the function
  $F$ defined in~\eqref{eq:F} is finite-valued only if
  $\lmval{i} \ge a \norm{(\xval{i}_{1:m-1}; 1)}^2$ for all $i \in
  \until{n}$.
\end{lemma}
\begin{proof}
  The proof mimics the steps of the proof of
  Lemma~\ref{le:finite-F-quad}.  Assume there exists $\tilde i \in
  \until{n}$ such that $\lmval{\tilde i} < a \norm{(\xval{\tilde
      i}_{1:m-1}; 1)}^2$.  For any $k$ such that $\data{k} \in
  \Xihat_{\tilde i}$, we have
  \begin{align*}
    g_k(\xval{\tilde i}, \lmval{\tilde i}, \xi) = a (\xi_m -
    (\xi_{1:m-1}; 1)^\top \xval{\tilde i})^2 - \lmval{\tilde i}
    \norm{\xi - \data{k}}^2. 
  \end{align*}
  Parametrizing $\xi = \alpha (-\xval{\tilde i}_{1:m-1} ; 1)$,
  we obtain
  \begin{align*}
    g_k(\xval{\tilde i}, \lmval{\tilde i}, \alpha (-\xval{\tilde
    i}_{1:m-1} ; 1) ) & = a \Bigl(\alpha \norm{-\xval{\tilde i}_{1:m-1}
    ; 1}^2 - \xval{\tilde i}_m \Bigr)^2 
    \\
    & - \lmval{\tilde i} \norm{\alpha
    (-\xval{\tilde i}_{1:m-1} ;1) - \data{k}}^2.
  \end{align*}
  This is a quadratic function in the parameter $\alpha$ and the
  coefficient of the second-order term is $\norm{ (-\xval{\tilde
  i}_{1:m-1}; 1)}^2 (a \norm{ (-\xval{\tilde
  i}_{1:m-1}; 1)}^2 - \lmval{\tilde i})$. This coefficient is positive 
  by the assumption stipulated above. Therefore, $\max_{\xi}
  g_k(\xval{\tilde i}, \lmval{\tilde i}, \xi) = + \infty$. 
  Since $\max_{\xi} g_k(\xval{v_k}, \lmval{v_k}, \xi) > - \infty$ for
  all $k \in \until{N}$, we
  conclude that $F(\xvec,\lmvec) = + \infty$.
\end{proof}

Guided by the above result, let
\begin{align}
  \CC \! := \! \real^{nd} \! \times \! \setdef{\lmvec \in \realnonnegative^n \! }{\! \lmval{i}
  \! \ge \! a \norm{(\xval{i}_{1:m-1};1)}^2, \forall i \! \in \! \until{n}}. \label{eq:C-least-sq}
\end{align}
As a consequence of Lemma~\ref{le:finite-F-least-sq}, the optimizers
of~\eqref{eq:stochastic-net-dist-opt} belong to $\CC$ and so, the
saddle points of $\Laug$ over the domain $(\real^{nd} \times
\realnonnegative^n) \times (\real^n \times \real^{nd})$ are contained
in the set $\CC \times (\real^n \times \real^{nd})$. Further, $\CC$ is
closed, convex with a nonempty interior and the function $\Laugt$ is
convex-concave on $\CC \times (\real^n \times \real^{nd} \times
\real^{mN})$. Finally, one can show that~\eqref{eq:min-max-exchange}
holds in this case. Using these facts in
Theorem~\ref{th:convergence-dyn3} yields the following result.

\begin{corollary}\longthmtitle{Convergence of trajectories of
    $\spLaugt$ for least squares
    problem}\label{cr:convergence-least-sq}
  Let $f$ be of the form~\eqref{eq:f-least-sq} and $\CC$ be given
  in~\eqref{eq:C-least-sq}.  Assume further that there exists a saddle
  point $(\xvec^*,\lmvec^*,\nu^*,\eta^*,\{(\xival{k})^*\})$ of
  $\Laugt$ satisfying $(\xvec^*,\lmvec^*) \in \intr(\CC)$. Then, the
  trajectories of~\eqref{eq:dyn3} starting in $\CC \times \real^n
  \times \real^{nd} \times \real^{mN}$ remain in this set and converge
  asymptotically to a saddle point of $\Laugt$. As a consequence, the
  $(\xvec,\lmvec)$ component of the trajectory converges to an
  optimizer of~\eqref{eq:stochastic-net-dist-opt}.
\end{corollary}

In this case too, the saddle-point dynamics~\eqref{eq:dyn3} is
amenable to distributed implementation,
cf. Remark~\ref{re:dist-imp-new}, as one can write $\CC = \Pi_{i=1}^n
\setdef{(x,\lm) \in \real^d \times \realnonnegative}{\lm \ge a
  \norm{(x_{1:m-1} ; 1)}^2}$.

\section{Simulations}\label{sec:sims}

Here we illustrate the application of the distributed
algorithm~\eqref{eq:dyn3} to find a data-driven solution for the
regression problem with quadratic loss function and an affine
predictor~\cite[Chapter 3]{BS-AJS:02}, commonly termed as the
least-squares problem. This problem shows up in many
  applications, for example, in distributed estimation and target
  tracking.  Assume $n=10$ agents with communication topology defined
by an undirected ring with additional edges $\{(1,4), (2,5), (3,7),
(6,10)\}$.
The weight of each edge is equal to one.  We consider data points of
the form $\data{k} = (\what^{k},\yhat^{k})$ consisting of the input
$\what^{k} \in \real^4$ and the output $\yhat^{k} \in \real$ pairs.
The objective is to find an affine predictor $x \in \real^5$ using the
dataset such that, ideally, for any new data point $\xi = (w,y)$, the
predictor $x^\top(w;1)$ is equal to $y$. One way of finding such a
predictor $x$ is to solve the following problem
\begin{align}\label{eq:reg}
  \inf_{x} \Eb_{\Pb} \Bigl[f(x,w,y)\Bigr]
\end{align}
where $\Pb$ is the probability distribution of the data $(w,y)$ and
$f:\real^5 \times \real^5 \to \real$ is the quadratic loss function,
i.e., $f(x,w,y) = (x^\top(w;1) - y)^2$, corresponding to the case
considered in  Section~\ref{se:least-squares}.

To find the data-driven solution, we assume that each agent in the
network has $30$ i.i.d samples of $(w,y)$ and hence $N = 300$ is the
total number of samples. The dataset is generated by assuming the
input vector $w$ having a standard multivariate normal distribution,
that is, zero mean and covariance as the identity matrix
$\identity{4}$. The output $y$ is assigned values $y = [1, 4, 3, 2]*w
+ v$ where $v$ is a random variable, uniformly distributed over the
interval $[-1,1]$. This defines completely the distribution $\Pb$ of
$(w,y)$.  Let 
$\rad = 0.05$.
This value is assumed to be computed by the agents beforehand. This
defines completely the distributed optimization
problem~\eqref{eq:stochastic-net-dist-opt}.

\begin{figure}
  \centering
  \subfloat[$\xvec$]{\includegraphics[width=0.65\linewidth]{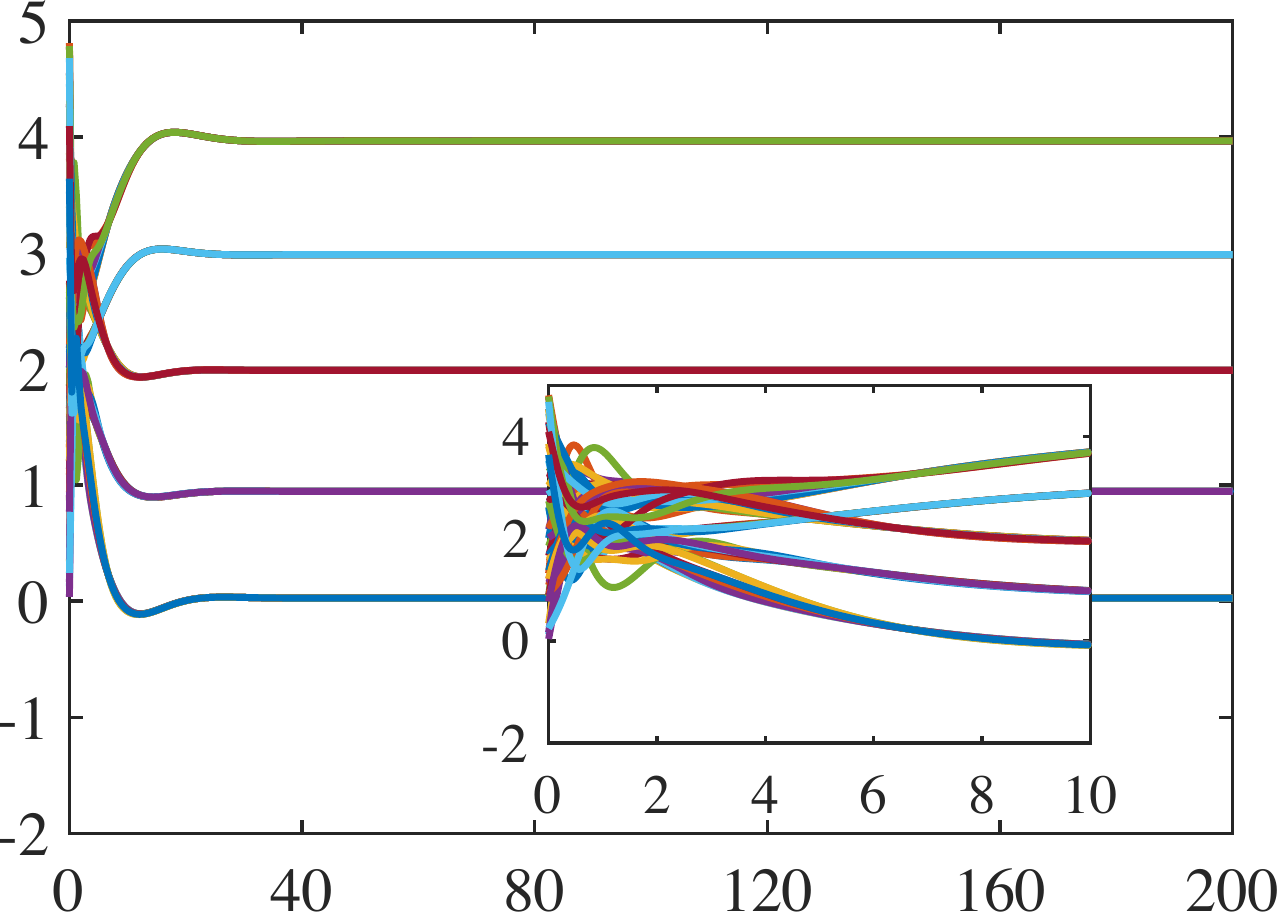}}
  \\
  \subfloat[$\lmvec$]{\includegraphics[width=0.65\linewidth]{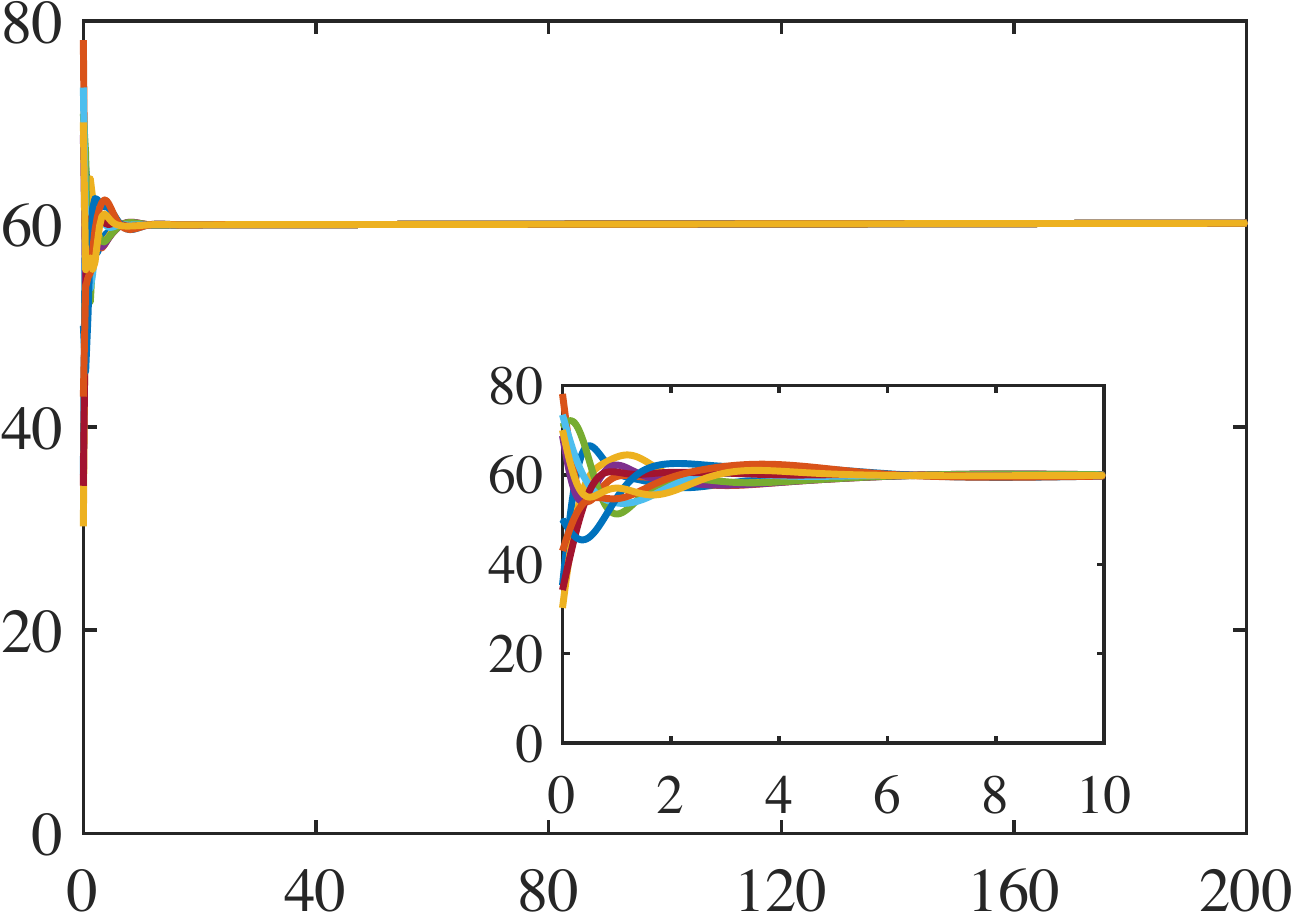}}
  \caption{Illustration of the execution of the
    dynamics~\eqref{eq:dyn3} to find a data-driven solution of the
    regression problem~\eqref{eq:reg}. Plots (a) and (b) depict the
    evolution of the primal variables of the distributed optimization
    problem~\eqref{eq:stochastic-net-dist-opt} defined
    for~\eqref{eq:reg} with $\rad = 0.05$ (for the sake of
    simplicity we have not shown the dual variables). The number of
    agents is $10$ and each agent collects $30$ i.i.d samples of the
    random variable. The initial condition $(\xvec(0),\lmvec(0))$ is
    chosen randomly from the set $[0,5]^{50} \times [30,80]^{10}$ and
    $\nu(0) = \zeros_{10}$, $\eta(0) = \zeros_{50}$, and $\xival{k}(0)
    = \zeros_5$ for all $k \in \until{N}$.  The primal variables
    converge to $\xvec = \ones_{10} \otimes x^*$ and $\lmvec = \lm^*
    \ones_{10}$ with $x^* = (0.9460; 3.9624; 2.9840; 1.9878; 0.0268)$  
    and $\lm^* = 58.4100$.  This is an equilibrium point of $\spLaugt$
    as well as an optimizer of~\eqref{eq:stochastic-net-dist-opt}. The
    agents reach consensus early in the execution. Therefore, we only
    observe $5$ (resp. $1$) curves in the plot of $\xvec$
    (resp. $\lmvec$), each corresponding to a component of $x$
    (resp. $\lm$). To depict this consensus clearly, we provided 
    the zoomed part of the initial $10$ 
    seconds for both $\xvec$ and $\lmvec$ plots. }\label{fig:one}
\end{figure}

Figure~\ref{fig:one} shows the execution of the distributed
algorithm~\eqref{eq:dyn3} that solves this problem. The trajectories
converge to an equilibrium of the dynamics~\eqref{eq:dyn3} whose
$(\xvec,\lmvec)$ component corresponds to an optimizer
of~\eqref{eq:stochastic-net-dist-opt}, consistent with
Corollary~\ref{cr:convergence-least-sq}. Furthermore, due to
  the projection operator in the dynamics, trajectories are contained
  in the set $\CC$ given in~\eqref{eq:C-least-sq} with $a =1$. Note
  that if one knows beforehand that
  $\lm^* \ge a \norm{(x^{*}_{1:m-1};1)}^2$ for some $a > 1$, then one
could further restrict the domain of the dynamics.

To evaluate the quality of the obtained solution, we compute the
average value of the loss function $f$ for a randomly generated
validation dataset consisting of $N_{\val} = 10^4$ data points
$\{(\what^{k}_{\val},\yhat^{k}_{\val})\}_{k=1}^{N_{\val}}$. These
points are i.i.d with the same distribution as that of the training
dataset generated above. Given the obtained solution $(\ones_{10}
\otimes x^*, \lm^* \ones_{10})$, see Figure~\ref{fig:one}, we evaluate
\begin{align}\label{eq:fval}
  f_{\mathrm{val}}^{N_{\val}}(x^*) = \frac{1}{N_{\val}}
  \sum_{k=1}^{N_{\val}} f(x^*,\what_{\val}^k,\yhat_{\val}^k)
\end{align}
and get $f_{\mathrm{val}}^{N_{\val}}(x^*) = 0.3387$. This is the
average loss for the solution $x^*$ obtained by the agents cooperating
with each other, essentially fusing the information of the $300$ data
points. Note that each agent individually can also solve a data-driven
solution with the samples gathered by it. However, the solution
obtained in such a manner, in general, incurs a higher average
loss. In the current setup, if agent $1$
solves~\eqref{eq:stochastic-net-c-opt} only with the data available to
it (and keeping other parameters equal), then it gets the optimizer as
$x^{\mathrm{opt,1}} = (0.8548; 3.8933; 2.8623; 2.1317; 0.2227)$.
Using the validation dataset, we obtain
$f_{\mathrm{val}}^{N_{\val}}(x^{\mathrm{opt,1}}) = 0.4520$, which is
significantly greater than $f_{\mathrm{val}}^{N_{\val}}(x^*)$. 
This shows the value of cooperation, that
is, fusing the information contained in the data available to
different agents leads to an optimizer with better out-of-sample
performance. To highlight this fact further, Figure~\ref{fig:loss}
shows the effect of the number of cooperating agents on the average
loss incurred by the obtained solution to the data-driven optimization
problem. As the plot shows, the improvement in performance due to
coordination becomes more prominent as the size of the coordination
agents grows.

\begin{figure}
  \centering
	  \includegraphics[width
    = 0.3 \textwidth]{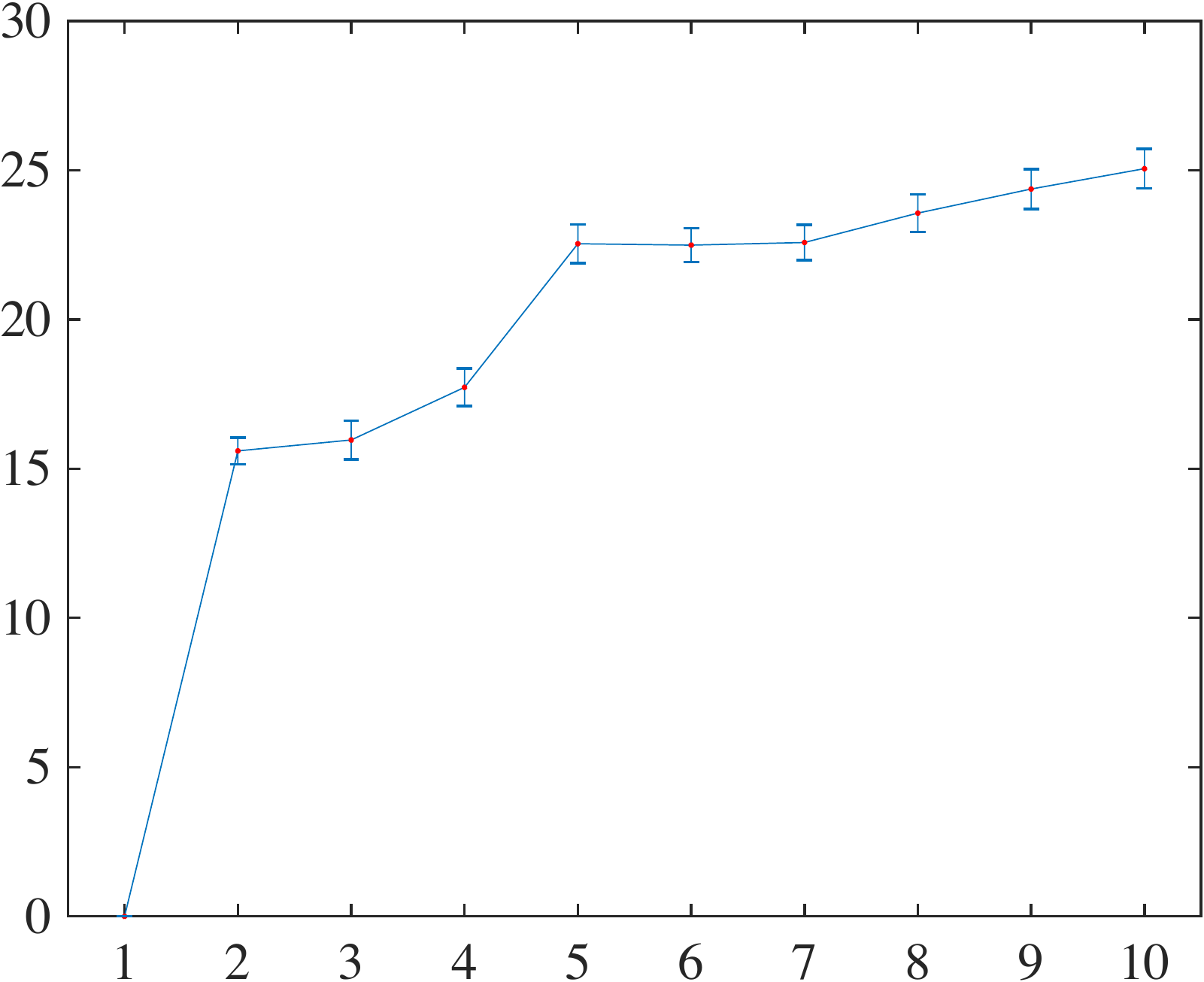}
    \caption{Relative benefit of cooperation versus isolation. The
      $x$-axis represents the number of agents cooperating in our
      example. The $y$-axis plots the function
      $R(i):=\frac{f_{\mathrm{val}}^{N_{\val}}(x^{\mathrm{opt,1}}) -
        f_{\mathrm{val}}^{N_{\val}}
        (x^{\mathrm{opt,i}})}{f_{\mathrm{val}}^{N_{\val}}
        (x^{\mathrm{opt,1}})} \times 100 \% $. Here, the function
      $f_{\mathrm{val}}^{N_{\val}}$ is given in~\eqref{eq:fval} and
      $x^{\mathrm{opt,i}}$ denotes the solution
      of~\eqref{eq:stochastic-net-c-opt} determined by $i$ cooperating
      agents. That is, $R(i)$ for each $i \in \until{n}$ represents
      the relative decrease in the average loss when $i$ agents
      cooperate, taking the base case as no cooperation ($i=1$).
      Thus, as more agents cooperate, the average loss decreases. The
      error bars represent the standard deviation of $R(i)$ values as
      we carry out $100$ runs of the simulation, randomly generating
      different validation data each time.
    } \label{fig:loss}
\end{figure}

\section{Conclusions}\label{sec:conclusions}

We have considered a cooperative stochastic optimization problem,
where a group of agents rely on their individually collected data to
collectively determine a data-driven solution with guaranteed
out-of-sample performance.  Our technical approach has proceeded by
first developing a reformulation in the form of a distributed
optimization problem, leading us to the identification of an augmented
Lagrangian function whose saddle points have a one-to-one
correspondence with the primal-dual optimizers.  This characterization
relies upon certain interchangeability properties between the min and
max operators. Our discussion has identified several classes of
objective functions for which these properties hold: convex-concave
functions, convex-convex functions quadratic in the data, and
convex-convex functions associated to least-squares problems.
Building on the analytical results, we have designed a distributed
saddle-point coordination algorithm where agents share their
individual estimates about the solution, not the collected data. We
have also formally established the asymptotic convergence of the
algorithm to the solution of the cooperative stochastic optimization
problem.  Future work will explore the characterization of the
algorithm convergence rate, the design of strategies capable of tracking the solution of the
stochastic optimization problem when new data becomes available in an
online fashion, and the analysis of scenarios with network chance
constraints.

\end{document}